\documentclass{article}
\usepackage[utf8]{inputenc}
\usepackage{Preamble}

\theoremstyle{plain}\newtheorem{theorem}{Theorem}[section]
\theoremstyle{plain}
\theoremstyle{plain}\newtheorem{lemma}[theorem]{Lemma}
\theoremstyle{plain}
\theoremstyle{definition}
\theoremstyle{definition}\newtheorem{remark}[theorem]{Remark}
\theoremstyle{definition}\newtheorem{remarks}[theorem]{Remarks}

\title{A priori $L^\infty-$bound for Ginzburg-Landau energy minimizers with divergence penalization}
\date\today

\begin{document}

\author[1]{Lia Bronsard\footnote{bronsard@mcmaster.ca }}
\author[1]{Andrew Colinet\footnote{andrew.colinet@alum.utoronto.ca}}
\author[1]{Dominik Stantejsky\footnote{stantejd@mcmaster.ca}}
\affil[1]{Department of Mathematics and Statistics, McMaster University, Hamilton, ON L8S 4L8 Canada}

\maketitle

\begin{abstract}
We consider minimizers $u_\varepsilon$ of the Ginzburg-Landau energy with quadratic divergence penalization on a simply-connected two-dimensional domain $\Omega$.
On the boundary, strong tangential anchoring is imposed.
We prove that minimizers satisfy a $L^\infty$-bound uniform in $\varepsilon$
when $\Omega$ has $C^{2,1}-$boundary and that the Lipschitz constant blows up like $\varepsilon^{-1}$ when $\Omega$ has $C^{3,1}-$boundary. Our theorem extends to $W^{2,p}-$regularity result for our elliptic system with mixed Dirichlet-Neumann boundary condition. \\
\linebreak
\textbf{Keywords:}  Ginzburg-Landau energy, regularity for elliptic systems, divergence penalization, mixed Dirichlet-Neumann boundary condition
\linebreak
\textbf{MSC2020:} 
49K20, 
35B45, 
35B38, 
35E20, 
49S05. 
\end{abstract}

\section{Introduction}
In this article we obtain $L^\infty$ a-priori estimates for minimizers of the Ginzburg-Landau energy with divergence penalization:
\begin{equation}\label{def:IntroEnergy}
    \int_{\Omega}\!{}\biggl\{
    \frac{1}{2}|\nabla{}u|^{2}
    +\frac{\ktwo}{2}\bigl(\text{div}(u)\bigr)^{2}
    +\frac{1}{4\varepsilon^{2}}\bigl(|u|^{2}-1\bigr)^{2}\biggr\}
    \, ,
\end{equation}
where $k>0$, $\Omega\subset\mathbb{R}^2$ is a domain with sufficiently smooth boundary and we impose tangential boundary conditions, namely $u\cdot \nu=0$ on $\partial \Omega$ with $\nu$ the exterior unit normal to $\partial \Omega$.
When $k=0$ the energy, \eqref{def:IntroEnergy}, reduces to the Ginzburg-Landau energy which has been the subject of intensive research, see for example \cite{BBH,MSZ,PR,San,Str} in the case of Dirichlet boundary conditions. 
The study of the formation of singularities when $\eps\to 0$ is very important in the
understanding of defects in superconductors, and an important first step is to obtain the uniform $L^\infty-$bound $\Vert u\Vert_{L^\infty}\leq 1$ as well as the degenerate Lipschitz bound $\varepsilon\Vert \nabla u\Vert_{L^\infty}\leq C$. 
In the case of the Ginzburg-Landau energy with Dirichlet boundary conditions, this is obtained via a classical maximum principle and a blow up procedure for the Lipschitz bound. However, as observed in \cite[Remark 1.1]{CoLa}, when an anisotropy is included in the energy, namely replacing $|\nabla{}u|^{2}$ with a more general function $W(x, \nabla u)$, for a positive definite quadratic form $W(x,\cdot)$, it becomes very difficult to obtain a-priori estimates.
This is due to the fact that in the anisotropic case, there is coupling of all components, and in fact one can only hope to obtain estimates such as $\Vert u\Vert_{L^\infty(\Omega)}\leq C\Vert u\Vert_{L^\infty(\partial\Omega)}$ with $C>1$, rather than $C=1$ as in the isotropic case without divergence. 
In fact, in \cite{Polya30} it is shown that $C=1$ cannot hold in presence of the divergence term and without the bulk potential, see also Figure \ref{fig:example_polya}.

\begin{figure}
\begin{center}
\includegraphics[scale=0.90]{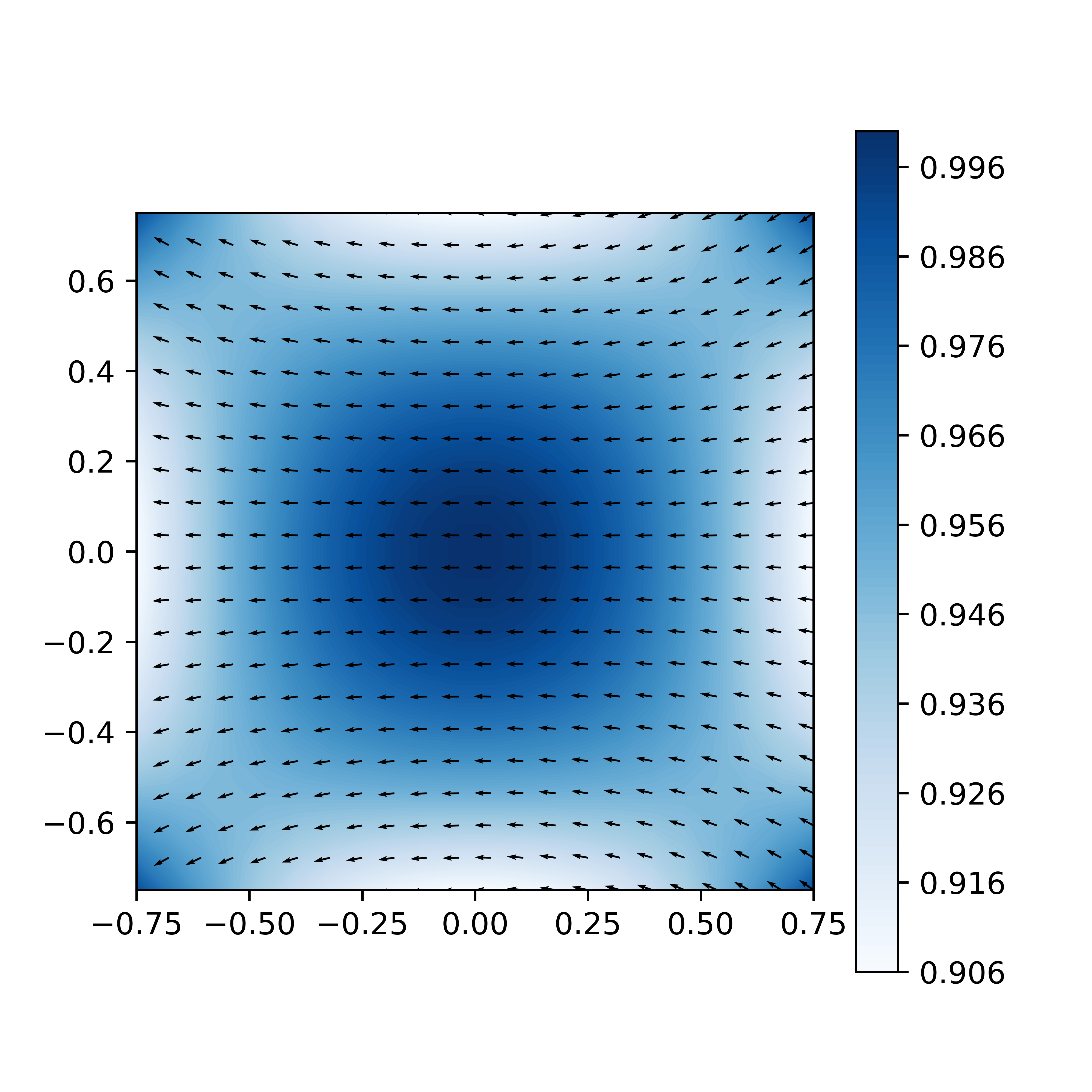} 
\caption{Adaptation of the example in \cite{Polya30} for dimension $n=2$.
For given $k\in\mathbb{R}$, the vector field $u(x,y)=(\frac{\alpha}{2}(x^2+y^2)-\beta, -\gamma xy)$ solves the equation $-\Delta u - k\nabla\dif u = 0$ for arbitrary $\beta,\gamma\in\mathbb{R}$ and $\alpha=\frac{k}{k+2}\gamma$ (if $k\neq -2$). 
If $k=-2$, then set $\gamma=0$ and $\alpha,\beta\in\mathbb{R}$ arbitrary. 
The above plot depicts the vector field $u(x,y)$ for $k=\beta=\gamma=1$ and $\alpha=\frac13$, the color indicates the norm $|u|$. 
One can see that if $u$ is restricted to a domain $\Omega$ chosen to be a small enough disk around $0$, then the maximum of $|u|$ occurs in the interior of the disk.}
\label{fig:example_polya}
\end{center}
\end{figure}

Further, close to the boundary one relies on the Dirichlet condition being regular and respecting the $L^\infty-$ bounds in order to get a uniform estimate.
Both assumptions are no longer satisfied in our setting of tangential anchoring. 
Imposing $u\cdot\nu=0$ leads to novel mixed Dirichlet and Neumann type boundary conditions \eqref{def:PDE}, as shown in \cite{ABB}.
In this article we will use a reflection-extension method compatible with these boundary conditions which preserve the character of the partial differential equation and allow us to apply classical regularity estimates \cite{Gi93}.
Using completely different methods, the article \cite{Si} obtains regularity for more general elliptic systems with mixed Dirichlet and curvature dependent Robin boundary conditions that do not, however, include our setting.

The energy \eqref{def:IntroEnergy} that we study, models thin ferroelectric smectic $C^*$ liquid crystals in the case that the 2-dimensional splay moduli is larger than the 2-dimensional bend moduli, and is studied in the article by Colbert-Kelly and Phillips \cite{CoKePh} with Dirichlet boundary conditions. 
They characterize the singular limit as $\eps\to 0$ of minimizers of the energy \eqref{def:IntroEnergy} as well as for the case that the bend moduli is larger than the splay moduli, namely when $|\dif(u)|^2$ is replaced by $|\curl(u)|^2$. 
Here, we study \eqref{def:IntroEnergy} with tangential boundary conditions, as they arise naturally in experiments in liquid crystal (see for example the work of Volovik and Lavrentovich \cite{LaVol} where nematic liquid drops are placed in an isotropic medium, allowing for the control of nematic boundary behaviour). The study of the singular limit of the Ginzburg-Landau functional with tangential boundary conditions was studied in \cite{ABB} and its $\Gamma-$limit in \cite{ABC}.

Tangential boundary conditions also appear naturally when studying the full Oseen-Frank energy for nematic liquid crystals with $\mathbb{S}^1-$valued $u$. 
Indeed, Day and Zarnescu \cite{DayZar} proved that in order to obtain finite energy global minimizers, one needs to impose special boundary conditions, the most natural being the tangential anchoring $u\cdot \nu=0$.

Our main result is the following:
\begin{theorem}\label{thm:main}
Let $\Omega\subset\mathbb{R}^2$ be open, bounded, and simply-connected with $C^{2,1}-$boundary. 
Then there exist constants $C_1,\varepsilon_{0}>0$, only depending on
$\ktwo$ and $\Omega$ such that for all
$\varepsilon\in (0,\varepsilon_{0})$ and any minimizers $u_\varepsilon$ of
\eqref{def:IntroEnergy} subject to the boundary condition $u\cdot\nu=0$, it holds
\begin{equation*}
\Vert u_\varepsilon\Vert_{L^\infty(\Omega)}
\ \leq \ 
C_1.
\end{equation*}
If $\Omega$ has $C^{3,1}$-boundary then there also exists a constant
$C_{2}>0$, only depending on $\ktwo$ and $\Omega$ such that for all
$\varepsilon\in(0,\varepsilon_{0})$
\begin{equation*}
\Vert \nabla u_\varepsilon\Vert_{L^\infty(\Omega)}
\ \leq \ 
\frac{C_2}{\varepsilon}.
\end{equation*}
\end{theorem}

\begin{remarks}\hspace{2pt}\\[-15pt]
    \begin{enumerate}
        \item
            As opposed to previous results \cite{BPP,BBO,CoKePh} on the case of Dirichlet boundary conditions, we avoid imposing regularity assumptions on the
            boundary data beyond being trace class.
            This is possible since our techniques relies on interior regularity methods.
            However, since the methods of \cite{CoKePh} require regularity
            of their boundary data these are inaccessible.
        \item
            Our energy includes an additional term penalizing divergence as in
            \cite{CoKePh}.
            The addition of this term prevents classical methods, as
            seen e.g.\ in Lemma $2.1$ of \cite{ABB}, for
            demonstrating $L^{\infty}-$regularity from working.
        \item
            Interestingly, the techniques employed here do not result in a
            symmetric result for curl penalization with tangential boundary
            conditions.
            This is because the minimizer for \eqref{def:IntroEnergy} also satisfies Neumann boundary conditions (see \eqref{def:PDE} in Subsection \ref{subsec:Notation}), while if we had penalized the curl, we would obtain Robin boundary conditions for the tangential component of $u$, while the normal component $u_\nu$ still satisfies the Dirichlet condition $u_\nu=0$.
            A general Robin boundary conditions can not be handled by our reflection approach, even in the case of a half space. 
            The method from \cite{Si} will prove to be useful in that case, see \cite{ABCSV}.
            However, the Robin boundary condition coming from the curl penalization includes a curvature term, which vanishes for flat boundaries, and the Robin conditions simplifies to a Neumann condition which we can therefore handle in that special case.
        \item  
            Replacing the strong anchoring condition $u\cdot\nu=0$ with a weak anchoring by adding a boundary penalization term such as $\int_{\partial\Omega}(u\cdot\nu)^2$ to the energy \eqref{def:IntroEnergy}, leads to a Neumann condition for $u_\tau$ and a Robin-type condition for $u_\nu$.
            In the case of weak anchoring combined with a curl-penalization, we end up with Robin boundary conditions for both $u_\nu$ and $u_\tau$.
        \item
            Many of the ideas of the proof extend to the case of a multiply-connected domain.
            However, for a complete proof, one would need an extension of
            the lower bound results of \cite{ABB} in order to demonstrate
            the uniform $L^{4}$ bounds of Lemma \ref{lem:SubseqBound} in
            the more general setting.
        \item
            The statement of Theorem \ref{thm:main} only discusses
            regularity up to Lipschitz order.
            However, the proof technique can also be used to demonstrate
            further regularity provided $\Omega$ has a sufficiently smooth
            boundary.
            More concretely, for $k\ge0$, if $\Omega$ has $C^{k+3,1}$ boundary
            then it is possible to demonstrate $C^{k,1}$ bounds of order
            $\varepsilon^{-(k+1)}$.
        \item  
            More generally, our method of extension-reflection can be used to prove regularity for a system of equations of the form
            \begin{align*}
            \begin{cases}
            -\Delta{}u -k\nabla\dif u \ = \ f &
            \text{on }\Omega,\\
            \hfill u_{\nu} \ = \ 0 &\text{on }\partial\Omega, \\
            \hfill \partial_{\nu}u_{\tau} \ = \ 0 &\text{on }\partial\Omega,
            \end{cases}
            \end{align*}
            for $f\in L^p(\Omega,\mathbb{R}^2)$, $p\in (1,\infty)$.
            This follows since one can replace the uniform bound on the $L^4-$norm of $u$ from Section \ref{sec:Regularity} in the proofs of Lemma \ref{lem:UniformBoundWidehatU} and Lemma \ref{lem:LipschtzBound} by the $L^p-$norm of $f$.
    \end{enumerate}
\end{remarks}

The main idea to prove Theorem \ref{thm:main} is to extend a minimizer $u$ across the domain's boundary.
In the absence of the divergence penalization, this has been done in \cite{Sv} for non-curved boundaries, and in \cite{DiMiPi} for curved boundaries with a Dirichlet boundary condition.
The new difficulties are that:
\begin{enumerate}
    \item
        We consider only boundary restrictions on the normal part.
        This is distinct from Dirichlet boundary conditions since the
        tangential part of our function is free along the boundary.
    \item
        We do not impose additional regularity assumptions on our boundary
        data beyond what can be obtained by the Trace theorem.
        Thus, we are only working with boundary data which is in
        $W^{\frac{1}{2},2}$.
\end{enumerate}

Our boundary regularity assumption implies that the curvature of $\partial\Omega$ is bounded and therefore satisfies an inner/outer ball condition, enabling us to extend the outward unit normal and therefore also a unit tangent vector field. 
These two vector fields will be used to define suitable tangent-normal coordinates in a tubular neighbourhood of the boundary $\partial\Omega$.
We can then express the extension $U$ of the minimizer $u$ in these coordinates.

From the (elliptic) Euler-Lagrange equations of the energy \eqref{def:IntroEnergy} satisfied by the minimizer $u$, we are able to show that the extended function $U$ also satisfies a PDE with appropriate boundary conditions.
Due to the curvature of the boundary, the differential operators on the exterior must be corrected by including an anisotropic metric quantifying the distortion induced by the reflection. 
In Appendix \ref{app:polar-example} we detail the instructive case of the unit disk where we use polar coordinates to compute the metric and distortion factors.
Furthermore, the system of PDEs on the inside and outside can be combined into a single system of PDEs, a process that we refer to as ``gluing of PDEs''.
This process is successful thanks to the compatibility of our boundary conditions with the elliptic operator and the reflection.
This allows us to treat any point $x_0$ on the boundary $\partial\Omega$ as an interior point for the domain of the glued PDE.

Once this is accomplished, we can establish a $L^4-$bound on $u$ and therefore also on the extension $U$, similar to \cite{ABC,ABB,BPP}.
It is of crucial importance that all estimates are independent of $\varepsilon$ in order to obtain constants $C_1,C_2$ in Theorem \ref{thm:main} that are independent of $\varepsilon$.

The final step, inspired by \cite{BPP,CoKePh}, is to rescale $U$ to a function $\hat{U}$ defined locally.
Using the PDE for $U$ and the uniform $L^4-$bounds established before, we can apply elliptic regularity theory for systems satisfying Legendre-Hadamard condition to obtain a uniform $L^\infty$ and Lipschitz bound for $\hat{U}$.
Using compactness of $\partial\Omega$, we deduce bounds for $U$, which imply that the original minimizer $u$ enjoys the bounds in Theorem \ref{thm:main}.

\paragraph{Outline of the paper:}
In Section \ref{sec:Preliminaries} we introduce the necessary notation to prove the main result, in particular the tangent-normal coordinates. 
The gluing of the interior and reflected PDE is carried out in Section \ref{sec:PDEGluing}.
Related to this section we include two appendices.
In Appendix \ref{app:calculations} we include, for completeness, the detailed calculations necessary to ``glue" the PDEs.
In Appendix \ref{app:polar-example} we present the instructive but somewhat simpler case of reflecting when $\Omega$ is the unit disk. 
In Section \ref{sec:Regularity} the uniform $L^4-$bound is derived which is a necessary ingredient together with the rescaling argument in Section \ref{sec:Rescaling} to conclude.

\section{Preliminaries}\label{sec:Preliminaries}
In this section, we go over some preliminaries necessary for the rest of the article.
In particular, we will use them to demonstrate in Section \ref{sec:PDEGluing}, that we may extend the solution to our elliptic PDE on an enlarged domain provided it satisfies suitable boundary conditions.

\subsection{Notation}\label{subsec:Notation}
For $x=(x_{1},x_{2})\in\mathbb{R}^{2}$ we define
$x^{\perp}\coloneqq(-x_{2},x_{1})$.
Throughout we let $\Omega\subset\mathbb{R}^{2}$ be a bounded
simply-connected domain with $C^{2,1}-$boundary.
Suppose $0<r<\text{inj}(\partial\Omega)$, where
$\text{inj}(\partial\Omega)$ denotes the injectivity radius of
$\partial\Omega$, and set
\begin{align*}
    \Omega_{r}&\coloneqq
    \{x\in\Omega:0<\text{dist}(x,\partial\Omega)<r\},\\
    \widetilde{\Omega}_{r}&\coloneqq
    \{x\in\mathbb{R}^{2}:0\le\text{dist}(x,\partial\Omega)<r\}.
\end{align*}
For a function $u\colon\Omega\to\mathbb{R}^{2}$,
with well-defined trace, we let $u_{\tau}$, $u_{\nu}$ denote, respectively,
the \emph{tangential part} of $u$ along
$\partial\Omega$ and the \emph{normal part} of $u$ along
$\partial\Omega$.
These are defined, respectively, by
\begin{equation*}
    u_{\tau}\coloneqq{}u\cdot\tau,
    \hspace{15pt}u_{\nu}\coloneqq{}u\cdot\nu
\end{equation*}
where $\tau$ is the locally defined unit tangent vector of positive
orientation and $\nu\coloneqq\tau^{\perp}$.
For $\Omega$ as above we introduce, see \cite{ISS} for a formal treatment,
the function spaces
\begin{align*}
    W_{T}^{1,2}(\Omega;\mathbb{R}^{2})
    &\coloneqq\bigl\{u\in{}W^{1,2}(\Omega;\mathbb{R}^{2}):u_{\nu}=0\bigr\},
\end{align*}
\noindent{}For
appropriate functions $\varphi\colon\Omega\to\mathbb{R}$ we let
\begin{equation*}
    \nabla^{\perp}\varphi\coloneqq(\nabla\varphi)^{\perp}.
\end{equation*}
Since we will be concerned with boundary behaviour we let, for $r>0$,
$B_{r,+}(0)$ denote
\begin{equation*}
    B_{r,+}(0)\coloneqq\{(x_{1},x_{2})\in\mathbb{R}^{2}:|(x_{1},x_{2})|<r,
    \hspace{5pt}x_{2}>0\}.
\end{equation*}
For $y\in\mathbb{R}$ we also use the notation
\begin{equation*}
    B_{r,+}(y,0)\coloneqq{}(y,0)+B_{r,+}(0).
\end{equation*}
We let $u_{\varepsilon}\in{}W_{T}^{1,2}(\Omega;\mathbb{R}^{2})$ be the minimizer
of the energy 
\begin{equation}\label{def:Energy}
    GL_{\varepsilon,\text{div}}(u,\Omega)\coloneqq
    \int_{\Omega}\!{}\biggl\{
    \frac{1}{2}|\nabla{}u|^{2}
    +\frac{\ktwo}{2}\bigl(\text{div}(u)\bigr)^{2}
    +\frac{1}{4\varepsilon^{2}}\bigl(|u|^{2}-1\bigr)^{2}\biggr\}.
\end{equation}
We will use, for a $\mathcal{L}^{2}$ measurable set $A\subseteq\mathbb{R}^{2}$,
$GL_{\varepsilon,\text{div}}(u,A)$ to denote this energy over the set $A$.
When no confusion will arise we will omit mention of the set from the notation.
We note that the minimizer satisfies, see equation $(2.3)$ of \cite{ABB},
\begin{equation}\label{def:PDE}
    \begin{cases}
        -\kone\Delta{}u_{\varepsilon}
        -\ktwo\nabla(\text{div}(u_{\varepsilon}))
        =\frac{u_{\varepsilon}}{\varepsilon^{2}}(1-|u_{\varepsilon}|^{2})&
        \text{on }\Omega,\\
        (u_{\varepsilon})_{\nu}(x)=0&\text{on }\partial\Omega,\\
        \partial_{\nu}(u_{\varepsilon})_{\tau}(x)=0&\text{on }\partial\Omega,
    \end{cases}
\end{equation}
in the sense that for all $v\in{}W_{T}^{1,2}(\Omega;\mathbb{R}^{2})$ we have
\begin{equation}\label{eq:WeakForm}
    \kone\int_{\Omega}\!{}\nabla{}u_{\varepsilon}:\nabla{}v
    +\ktwo\int_{\Omega}\!{}\bigl(\text{div}(u_{\varepsilon})\bigr)
    \bigl(\text{div}(v)\bigr)
    =\int_{\Omega}\!{}\frac{u_{\varepsilon}\cdot{}v}{\varepsilon^{2}}
    \bigl(1-|u_{\varepsilon}|^{2}\bigr).
\end{equation}
Note that here we define
$\partial_{\nu}(u_{\varepsilon})_{\tau}\in{}\bigl(W^{\frac{1}{2},2}(\partial\Omega)\bigr)^{*}$ by
\begin{equation}\label{def:NormalDeriv}
    \bigl<\partial_{\nu}(u_{\varepsilon})_{\tau},\varphi\bigr>
    \coloneqq
    -\kone\int_{\Omega}\!{}\nabla{}u_{\varepsilon}:\nabla{}\overline{\varphi}
    -\ktwo\int_{\Omega}\!{}\bigl(\text{div}(u_{\varepsilon})\bigr)
    \bigl(\text{div}(\overline{\varphi})\bigr)
    +\int_{\Omega}\!{}
    \frac{u_{\varepsilon}\cdot\overline{\varphi}}{\varepsilon^{2}}
    \bigl(1-|u_{\varepsilon}|^{2}\bigr)
\end{equation}
where $\overline{\varphi}\in{}W_{T}^{1,2}(\Omega;\mathbb{R}^{2})$
satisfies $\overline{\varphi}_{\tau}=\varphi$.
Observe that \eqref{def:NormalDeriv} is independent of the choice of extension
since the definition is linear  and if
$\overline{\varphi}_{1},\overline{\varphi}_{2}\in{}W_{T}^{1,2}(\Omega;\mathbb{R}^{2})$ both extend $\varphi$ then
$\overline{\varphi}_{1}-\overline{\varphi}_{2}\in{}W_{0}^{1,2}(\Omega;\mathbb{R}^{2})$.

\subsection{Coordinates}\label{subsec:Coordinates}
We introduce tangent-normal coordinates similar to the ones constructed in
Subsection $1.2$ of \cite{ABC} in which all further details are provided.\\

We parametrize $\partial\Omega$ by its arclength, $L$, using a
$C^{2,1}$ curve
$\gamma=(\gamma_{1}(y_{1}),\gamma_{2}(y_{1}))$
where $\gamma\colon\mathbb{R}\slash{}L\mathbb{Z}\to\partial\Omega$.
We define the unit tangent and unit inward normal vectors, $\tau$ and
$\nu$, respectively, so that $\{\tau,\nu\}$ is positively oriented and
$\nu=\tau^{\perp}$.
We have
\begin{equation}\label{eq:DerivativeIdentities}
    \tau'(y_{1})=\kappa(y_{1})\nu(y_{1}),\hspace{15pt}
    \nu'(y_{1})=-\kappa(y_{1})\tau(y_{1}),
\end{equation}
where $\kappa\colon\mathbb{R}\slash{}L\mathbb{Z}\to\mathbb{R}$
is the signed curvature of $\partial\Omega$.
Next we will define a coordinate chart on $\Omega$ which extends to a larger domain.
For $r_{0}\in\bigl(0,\frac{\text{inj}(\partial\Omega)}{4}\bigr)$ we define a $C^{1,1}$ map
$X\colon(\mathbb{R}\slash{}L\mathbb{Z})\times(0,r_{0})\to
\Omega_{r_{0}}$
by
\begin{equation}\label{TangentNormal:Def}
    X(y_{1},y_{2})\coloneqq\gamma(y_{1})+y_{2}\nu(y_{1}).
\end{equation}
We observe that by \eqref{eq:DerivativeIdentities} we have
\begin{equation}\label{PartialsOfX}
    \frac{\partial{}X}{\partial{}y_{1}}
    =\bigl[1-y_{2}\kappa(y_{1})\bigr]\tau(y_{1}),\hspace{15pt}
    \frac{\partial{}X}{\partial{}y_{2}}
    =\nu(y_{1}).
\end{equation}
From \eqref{PartialsOfX} we have
\begin{equation*}
    JX(y)\coloneqq\det(\nabla{}X(y))=1-y_{2}\kappa(y_{1}).
\end{equation*}
By perhaps shrinking $r_{0}$ we may ensure $JX$ is bounded away from zero on
$(\mathbb{R}\slash{}L\mathbb{Z})\times(0,r_{0})$.
Using $X$ and choosing $r_{1}\in(0,r_{0})$ chosen sufficiently small,
we can form an atlas of coordinate charts
$\bigl\{(\mathcal{U}_{j},\psi_{j})\bigr\}_{j=1}^{N}$ about
$\partial\Omega$ covering a tubular neighbourhood of $\partial\Omega$ in $\Omega$ where
$\psi_{j}\colon{}B_{r_{1},+}(0)\to\mathcal{U}_{j}\cap\Omega$.
\ds{
We can choose $b_j\in\mathbb{R}/L\mathbb{Z}$ (corresponding to points on the boundary $\gamma(b_j)\in\partial\Omega$) such that one has $\mathcal{U}_j\cap\Omega=X(B_{r_1}(b_j,0))$}.
In addition, the coordinate charts can be chosen to satisfy
\begin{align*}
    \nabla\psi_{j}(y_{1},y_{2})&=
    \begin{bmatrix}
        (1-y_{2}\kappa(y_{1}))\tau(y_{1})& \nu(y_{1})
    \end{bmatrix},\\
    (\nabla\psi_{j}^{-1})(x)&=
    \begin{bmatrix}
        \frac{1}{1-(\psi_{j}^{-1})^{2}(x)\kappa((\psi_{j}^{-1})^{1}(x))}
        \tau^{T}\bigl((\psi_{j}^{-1})^{1}(x)\bigr)\\
        \nu^{T}\bigl((\psi_{j}^{-1})^{1}(x)\bigr)
    \end{bmatrix}.
\end{align*}
We may also adjoin
$\mathcal{U}_{0}\coloneqq\{x\in\Omega:\text{dist}(x,\partial\Omega)
>\frac{r_{1}}{4}\}$
paired with the identity map $\psi_{0}=\text{id}$ to obtain an atlas for
$\Omega$.
\\

Next, we extend $X$ to a map
$\widetilde{X}\coloneqq(\mathbb{R}\slash{}L\mathbb{Z})\times(-r_{0},r_{0})
\to\widetilde{\Omega}_{r_{0}}$ in a tubular neighbourhood of $\partial\Omega$ by the same definition as in
\eqref{TangentNormal:Def}, see also Figure \ref{fig:charts}.
\ds{Using this extension, we can define $\widetilde{\mathcal{U}}_j:=\widetilde{X}(B_{r_1}(b_j,0))$.}
This allows us to view each $\psi_{j}$, for $j=1,2,\ldots,N$, as the restriction to $B_{r_{1},+}(0)$ of
a map $\widetilde{\psi}_{j}\colon{}B_{r_{1}}(0)\to\widetilde{\mathcal{U}}_{j}$.
We can also arrange that the chart $\widetilde{\psi}_{j}$ satisfies, for $(y_1,s)\in\psi_{j}^{-1}(\mathcal{U}_{j})$, that
\begin{equation*}
    \text{dist}(\widetilde{\psi}_{j}(y_1,-s),\partial\Omega)
    =\text{dist}(\widetilde{\psi}_{j}(y_1,s),\partial\Omega).
\end{equation*}
We may pair $\widetilde{\mathcal{U}}_{0}=\mathcal{U}_{0}$
with $\widetilde{\psi}_{0}\coloneqq{}\psi_{0}$ in order to
extend the atlas to one for $\widetilde{\Omega}_{r_{1}}$.
%
For the coordinate charts, $\{\widetilde{\psi}_{j}\}_{j=0}^{N}$, we let
$\widetilde{u}_{j}$ denote the coordinate representation in the $j^{\text{th}}$ coordinate map.
That is, we let
\begin{equation*}
    \widetilde{u}_{j}\coloneqq{}u\circ\widetilde{\psi}_{j}.
\end{equation*}
We also define, for $u$ satisfying \eqref{def:PDE},
$U\colon\widetilde{\Omega}\to\mathbb{R}^{2}$ by
\begin{equation}\label{def:uExtension}
    U(x)\coloneqq
    \begin{cases}
        u(x)&\text{for }x\in\Omega,\\
        \widetilde{u}_{\tau}\bigl(R\widetilde{\psi}_{j}^{-1}(x)\bigr)
        \tau\bigl((\widetilde{\psi}_{j}^{-1})^{1}(x)\bigr)
        -\widetilde{u}_{\nu}\bigl(R\widetilde{\psi}_{j}^{-1}(x)\bigr)
        \nu\bigl((\widetilde{\psi}_{j}^{-1})^{1}(x)\bigr)&
        \text{for }x\in\widetilde{\mathcal{U}}_{j}\setminus\Omega,
    \end{cases}
\end{equation}
where
$R\coloneqq\mathbf{e}_{1}\mathbf{e}_{1}^{T}-\mathbf{e}_{2}\mathbf{e}_{2}^{T}$.
Grosso modo, multiplying a vector by this matrix $R$ results in a switch of sign for the normal component, while preserving the tangential component.
A central tool will later be this transformation $R$ in the original variables on the domain $\Omega$. 
This reflection ensures continuity of the tangential component as well as the normal derivative of the extended function $U$, thus preserving the $W^{1,2}-$regularity and the boundary conditions from \eqref{def:PDE}.

\begin{figure}
\begin{center}
\begin{tikzpicture}[scale=1]
\pgfmathsetmacro{\r}{1} 	
\pgfmathsetmacro{\R}{3} 	
\pgfmathsetmacro{\dh}{1} 	
\pgfmathsetmacro{\dv}{0} 	
\pgfmathsetmacro{\a}{0.94} 	
\pgfmathsetmacro{\p}{0.2} 	
\pgfmathsetmacro{\ts}{90} 	


\draw[fill=blue!35!white, line width=0.5, dashed, domain=0:360, variable=\t] plot[samples=200] 
( {1.5*\R + \r*sqrt(\a*\a*cos(2*\t)+sqrt(1 - \a*\a*sin(2*\t)*sin(2*\t)))*sin(\t)
+(-\p)*(-sin(\t)*sqrt(\a*\a*cos(2*\t)+sqrt(1 - \a*\a*sin(2*\t)*sin(2*\t))) + cos(\t)*((-2*\a*\a*sin(2*\t) + (-4*\a*\a*sin(2*\t)*cos(2*\t))/(2*sqrt(1-\a*\a*sin(2*\t)*sin(2*\t))))/(2*sqrt(\a*\a*cos(2*\t)+sqrt(1 - \a*\a*sin(2*\t)*sin(2*\t))))))/sqrt(\a*\a*cos(2*\t)+sqrt(1 - \a*\a*sin(2*\t)*sin(2*\t)) + pow((-2*\a*\a*sin(2*\t) + (-4*\a*\a*sin(2*\t)*cos(2*\t))/(2*sqrt(1-\a*\a*sin(2*\t)*sin(2*\t))))/(2*sqrt(\a*\a*cos(2*\t)+sqrt(1 - \a*\a*sin(2*\t)*sin(2*\t)))),2))
} ,
{\r*sqrt(\a*\a*cos(2*\t)+sqrt(1 - \a*\a*sin(2*\t)*sin(2*\t)))*cos(\t)
-(-\p)*(cos(\t)*sqrt(\a*\a*cos(2*\t)+sqrt(1 - \a*\a*sin(2*\t)*sin(2*\t))) + sin(\t)*((-2*\a*\a*sin(2*\t) + (-4*\a*\a*sin(2*\t)*cos(2*\t))/(2*sqrt(1-\a*\a*sin(2*\t)*sin(2*\t))))/(2*sqrt(\a*\a*cos(2*\t)+sqrt(1 - \a*\a*sin(2*\t)*sin(2*\t))))))/sqrt(\a*\a*cos(2*\t)+sqrt(1 - \a*\a*sin(2*\t)*sin(2*\t)) + pow((-2*\a*\a*sin(2*\t) + (-4*\a*\a*sin(2*\t)*cos(2*\t))/(2*sqrt(1-\a*\a*sin(2*\t)*sin(2*\t))))/(2*sqrt(\a*\a*cos(2*\t)+sqrt(1 - \a*\a*sin(2*\t)*sin(2*\t)))),2))
} );

\draw[fill=blue!75!white] ({1.5*\R + \r*sqrt(\a*\a*cos(2*\ts)+sqrt(1 - \a*\a*sin(2*\ts)*sin(2*\ts)))*sin(\ts)} ,
{\r*sqrt(\a*\a*cos(2*\ts)+sqrt(1 - \a*\a*sin(2*\ts)*sin(2*\ts)))*cos(\ts)} ) circle[radius=0.1];

\draw[fill=blue!20!white, line width=1, domain=0:360, variable=\t] plot[samples=200] 
( {1.5*\R + \r*sqrt(\a*\a*cos(2*\t)+sqrt(1 - \a*\a*sin(2*\t)*sin(2*\t)))*sin(\t)} ,
{\r*sqrt(\a*\a*cos(2*\t)+sqrt(1 - \a*\a*sin(2*\t)*sin(2*\t)))*cos(\t)} 
);

\draw[] ({1.5*\R + \r*sqrt(\a*\a*cos(2*\ts)+sqrt(1 - \a*\a*sin(2*\ts)*sin(2*\ts)))*sin(\ts)} ,
{\r*sqrt(\a*\a*cos(2*\ts)+sqrt(1 - \a*\a*sin(2*\ts)*sin(2*\ts)))*cos(\ts)} ) circle[radius=0.1];

\draw[fill=white, line width=0.5, dotted, domain=0:360, variable=\t] plot[samples=200] 
( {1.5*\R + \r*sqrt(\a*\a*cos(2*\t)+sqrt(1 - \a*\a*sin(2*\t)*sin(2*\t)))*sin(\t)
+\p*(-sin(\t)*sqrt(\a*\a*cos(2*\t)+sqrt(1 - \a*\a*sin(2*\t)*sin(2*\t))) + cos(\t)*((-2*\a*\a*sin(2*\t) + (-4*\a*\a*sin(2*\t)*cos(2*\t))/(2*sqrt(1-\a*\a*sin(2*\t)*sin(2*\t))))/(2*sqrt(\a*\a*cos(2*\t)+sqrt(1 - \a*\a*sin(2*\t)*sin(2*\t))))))/sqrt(\a*\a*cos(2*\t)+sqrt(1 - \a*\a*sin(2*\t)*sin(2*\t)) + pow((-2*\a*\a*sin(2*\t) + (-4*\a*\a*sin(2*\t)*cos(2*\t))/(2*sqrt(1-\a*\a*sin(2*\t)*sin(2*\t))))/(2*sqrt(\a*\a*cos(2*\t)+sqrt(1 - \a*\a*sin(2*\t)*sin(2*\t)))),2))
} ,
{\r*sqrt(\a*\a*cos(2*\t)+sqrt(1 - \a*\a*sin(2*\t)*sin(2*\t)))*cos(\t)
-\p*(cos(\t)*sqrt(\a*\a*cos(2*\t)+sqrt(1 - \a*\a*sin(2*\t)*sin(2*\t))) + sin(\t)*((-2*\a*\a*sin(2*\t) + (-4*\a*\a*sin(2*\t)*cos(2*\t))/(2*sqrt(1-\a*\a*sin(2*\t)*sin(2*\t))))/(2*sqrt(\a*\a*cos(2*\t)+sqrt(1 - \a*\a*sin(2*\t)*sin(2*\t))))))/sqrt(\a*\a*cos(2*\t)+sqrt(1 - \a*\a*sin(2*\t)*sin(2*\t)) + pow((-2*\a*\a*sin(2*\t) + (-4*\a*\a*sin(2*\t)*cos(2*\t))/(2*sqrt(1-\a*\a*sin(2*\t)*sin(2*\t))))/(2*sqrt(\a*\a*cos(2*\t)+sqrt(1 - \a*\a*sin(2*\t)*sin(2*\t)))),2))
} );





\node[] at (1.5*\R,-0.8*\r) {$\Omega$};

\draw[line width=1,->] (2.0*\R+\dh/2,\r/2) to [out=150,in=30] (1.5*\R+\dh,\r/2);
\node[] at (1.75*\R+3*\dh/4,\r/2+0.4) {$X$};


\draw[line width=1,->] (2*\R+\dh/2,-\r/2) to [out=-150,in=-30] (1.5*\R+\dh,-\r/2);
\node[] at (1.75*\R+3*\dh/4,-\r/2-0.4) {$\widetilde{X}$};


\fill[blue!20!white] (2*\R+\dh,0) rectangle (3.1*\R+\dh,1*\r);

\fill[blue!35!white] (2*\R+\dh,-1*\r) rectangle (3.1*\R+\dh,0);

\draw[] ({2.5*\R+\dh-\r/2},0) -- ({2.5*\R+\dh+\r/2},0) arc (0:180:{\r/2}) -- cycle;

\draw[fill=blue!75!white] ({2.5*\R+\dh-\r/2},0) -- ({2.5*\R+\dh+\r/2},0) arc (0:-180:{\r/2}) -- cycle;

\draw[line width=1] (2*\R+\dh,0)--(3.1*\R+\dh,0);
\draw[->] (3.1*\R+\dh-0.1,0)--(3.1*\R+\dh+0.05,0) node[below] {$y_1$};
\draw[-] (2.5*\R+\dh,-0.1)--(2.5*\R+\dh,0.1) node[above=-0.1cm] {$b_j$};
\draw[-] (3*\R+\dh,-0.1)--(3*\R+\dh,0.1) node[above=-0.05cm] {$L$};

\draw[->] (2.1*\R+\dh,-1*\r)--(2.1*\R+\dh,1.2*\r) node[above] {$y_2$};
\draw[dotted, line width=0.5] (3.1*\R+\dh,\r)--(2*\R+\dh,\r);
\node[] at (2.1*\R+\dh+0.2,\r-0.2) {$r_0$};
\draw[dashed, line width=0.5] (3.1*\R+\dh,-\r)--(2*\R+\dh,-\r);
\node[] at (2.1*\R+\dh+0.2,-\r-0.2) {$-r_0$};

\end{tikzpicture}

\caption{Illustration of the construction of $X$ and the extension $\widetilde{X}$.}
\label{fig:charts}
\end{center}
\end{figure}
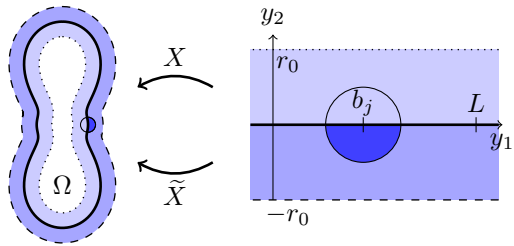

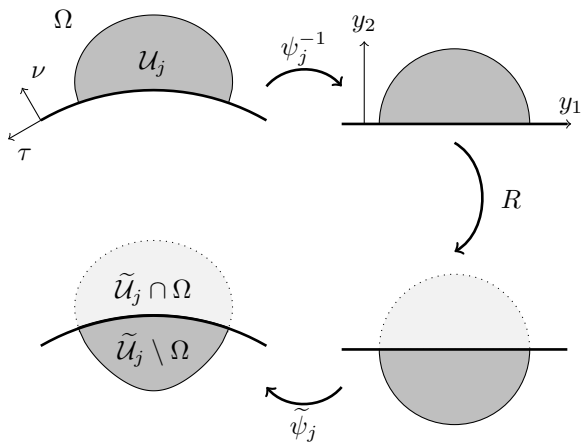
\begin{figure}
\begin{center}
\begin{tikzpicture}[scale=1]
\pgfmathsetmacro{\r}{1} 	
\pgfmathsetmacro{\R}{3} 	
\pgfmathsetmacro{\dh}{1} 	
\pgfmathsetmacro{\dv}{0} 	

\draw[fill=gray!50!white, domain=70.5:109.5, variable=\t] plot[samples=100] 
({1.5*\R+\R*cos(\t)+(sign(\r*\r - pow(\R*cos(\t),2))*sqrt(abs(\r*\r - pow(\R*cos(\t),2))))*cos(\t)}, 
{-0.85*\R+\R*sin(\t)+(sign(\r*\r - pow(\R*cos(\t),2))*sqrt(abs(\r*\r - pow(\R*cos(\t),2))))*sin(\t)}) -- cycle;
\draw[fill=white, line width=1, domain=60:120, variable=\t] plot ({1.5*\R+\R*cos(\t)}, {-0.85*\R+\R*sin(\t)});
\draw[->] ({1.5*\R+\R*cos(120)}, {-0.85*\R+\R*sin(120)}) -- ({1.5*\R+\R*cos(120)+(0.5)*cos(120)}, 
{-0.85*\R+\R*sin(120)+(0.5)*sin(120)}) node[above right] {$\nu$};
\draw[->] ({1.5*\R+\R*cos(120)}, {-0.85*\R+\R*sin(120)}) -- ({1.5*\R+\R*cos(120)-(0.5)*sin(120)}, 
{-0.85*\R+\R*sin(120)+(0.5)*cos(120)}) node[below right] {$\tau$};
\node[] at (1.1*\R,1.4*\r) {$\Omega$};
\node[] at (1.5*\R,0.8*\r) {$\mathcal{U}_j$};

\draw[line width=1,->] (2*\R,\r/2) to [out=50,in=130] (2*\R+\dh,\r/2);
\node[] at (2*\R+\dh/2,\r/2+0.5) {$\psi_j^{-1}$};
\draw[fill=gray!50!white] ({2.5*\R+\dh},0) -- ({2.5*\R+\dh+\r},0) arc (0:180:{\r}) -- cycle;
\draw[line width=1] (2*\R+\dh,0)--(3*\R+\dh,0);
\draw[->] (3*\R+\dh-0.1,0)--(3*\R+\dh+0.05,0) node[above] {$y_1$};
\draw[->] (2.1*\R+\dh,0)--(2.1*\R+\dh,\r+0.1) node[above] {$y_2$};

\draw[line width=1,->] (2.5*\R+\dh,{-0.25}) to [out=-30,in=30] (2.5*\R+\dh,{-\R/2-\dv-0.2});
\node[] at ({2.5*\R+\dh+0.75},{-\R/4-\dv/2-0.25}) {$R$};
\draw[fill=gray!50!white] ({2.5*\R+\dh},{-\R-\dv}) -- ({2.5*\R+\dh+\r},{-\R-\dv}) arc (0:-180:{\r}) -- cycle;
\draw[dotted,fill=gray!10!white] ({2.5*\R+\dh},{-\R-\dv}) -- ({2.5*\R+\dh+\r},{-\R-\dv}) arc (0:180:{\r}) -- cycle;
\draw[line width=1] (2*\R+\dh,{-\R-\dv})--(3*\R+\dh,{-\R-\dv});

\draw[line width=1,->] (2*\R+\dh,{-\R-\dv-\r/2}) to [out=-130,in=-50] (2*\R,{-\R-\dv-\r/2});
\node[] at (2*\R+\dh/2,{-\R-\dv-\r/2-0.5}) {$\widetilde{\psi}_j$};
\draw[dotted, fill=gray!10!white, domain=70.5:109.5, variable=\t] plot[samples=100] 
({1.5*\R+\R*cos(\t)+(sign(\r*\r - pow(\R*cos(\t),2))*sqrt(abs(\r*\r - pow(\R*cos(\t),2))))*cos(\t)}, 
{-\R-\dv-0.85*\R+\R*sin(\t)+(sign(\r*\r - pow(\R*cos(\t),2))*sqrt(abs(\r*\r - pow(\R*cos(\t),2))))*sin(\t)}) -- cycle;
\draw[fill=gray!50!white, domain=70.525:109.475, variable=\t] plot[samples=100] 
({1.5*\R+\R*cos(\t)-(sign(\r*\r - pow(\R*cos(\t),2))*sqrt(abs(\r*\r - pow(\R*cos(\t),2))))*cos(\t)}, 
{-\R-\dv-0.85*\R+\R*sin(\t)-(sign(\r*\r - pow(\R*cos(\t),2))*sqrt(abs(\r*\r - pow(\R*cos(\t),2))))*sin(\t)}) arc (109.475:70.525:{\R});
\draw[line width=1, domain=60:120, variable=\t] plot ({1.5*\R+\R*cos(\t)}, {-\R-\dv-0.85*\R+\R*sin(\t)});

\node[] at (1.5*\R,-\R-\dv+0.8*\r) {$\widetilde{\mathcal{U}}_j\cap\Omega$};
\node[] at (1.5*\R,-\R-\dv+0.0*\r) {$\widetilde{\mathcal{U}}_j\setminus\Omega$};
\end{tikzpicture}

\caption{We use the unit normal $\nu$ and tangent $\tau$ to parameterize the interior and exterior of the domain $\Omega$ via charts $\widetilde{\psi}_{j}$, in which the reflection $R$ can be represented by a simple sign change of $y_2$.}
\label{fig:coordinates}
\end{center}
\end{figure}

\section{PDE Gluing}
\label{sec:PDEGluing}

In this section we show that the extended function from
\eqref{def:uExtension} weakly satisfies a PDE similar to \eqref{eq:WeakForm} in $\widetilde{\mathcal{U}}_{j}$.
While this is typically done by establishing a PDE in local coordinates which flatten out the boundary into a half-ball, we instead demonstrate that a PDE is satisfied in the reflected $\widetilde{U}_j$ (i.e. in subsets of $\widetilde{\Omega}$) for each $j$ for which the boundary is still curved. 
This makes the rescaling approach for the PDE in Section \ref{sec:Rescaling} notationally simpler.
Note that the approach of extending the solution of a PDE has been done before for the one-constant approximation of the Landau-de Gennes functional with Dirichlet boundary condition (see Section 2.2 of \cite{DiMiPi}).
We extend this approach to handle the additional divergence term and the new mixed Dirichlet-Neumann boundary condition.

Before we state the main lemma of this section we introduce a bit of notation
that will be necessary.
For each $j=1,2,\ldots,N$ let
$\sigma_{j}\colon\widetilde{\mathcal{U}_{j}}\to\widetilde{\mathcal{U}}_{j}$
denote the function defined by
\begin{equation*}
    \sigma_{j}(x)\coloneqq
    \begin{cases}
        x&\text{for }x\in\widetilde{\mathcal{U}}_{j}\cap\Omega\\
        \widetilde{\psi}_{j}\bigl(R\widetilde{\psi}_{j}^{-1}(x)\bigr)&\text{for }
        x\in\widetilde{\mathcal{U}}_{j}\setminus\Omega.
    \end{cases}
\end{equation*}
This function can be used on the enlarged coordinate chart $\widetilde{\mathcal{U}}_{j}$ to pass from an exterior point of $\Omega$ into its interior, see Figure \ref{fig:coordinates}, while leaving interior points invariant.
We also introduce the function $\mathfrak{R}_{j}$, defined on
$\widetilde{\mathcal{U}}_{j}\times\mathbb{R}^{2}$ by
\begin{equation*}
    \mathfrak{R}_{j}(x,z)\coloneqq
    \Bigl[\tau\bigl((\widetilde{\psi}_{j}^{-1})^{1}(x)\bigr)
    \tau\bigl((\widetilde{\psi}_{j}^{-1})^{1}(x)\bigr)^{T}-
    \nu\bigl((\widetilde{\psi}_{j}^{-1})^{1}(x)\bigr)
    \nu\bigl((\widetilde{\psi}_{j}^{-1})^{1}(x)\bigr)
    \Bigr]z,
\end{equation*}
which corresponds to the aforementioned matrix $R$ written in the original variables.
We observe that
    \begin{equation*}
        \nabla\sigma_{j}(\sigma_{j}(x))\nabla\sigma_{j}(\sigma_{j}(x))^{T}=
        \begin{cases}
            I_{2}&\text{for }x\in\widetilde{\mathcal{U}}_{j}\cap\Omega,\\
            \nabla\widetilde{\psi}_{j}\bigl(\widetilde{\psi}_{j}^{-1}(x)\bigr)
            \mathcal{M}(x)
            \nabla\widetilde{\psi}_{j}\bigl(\widetilde{\psi}_{j}^{-1}(x)\bigr)^{T}&\text{for }x\in\widetilde{\mathcal{U}}_{j}\setminus\Omega
        \end{cases}
    \end{equation*}
    where
    \begin{equation*}
        \mathcal{M}(x)\coloneqq
        \begin{pmatrix}
                \frac{1}{\bigl(1-|(\widetilde{\psi}_{j}^{-1})^{2}(x)|
                \kappa\bigl((\widetilde{\psi}_{j}^{-1})^{1}(x)\bigr)\bigr)^{2}}& 0\\
                0& 1
        \end{pmatrix}.
    \end{equation*}
    Corresponding to this matrix we introduce the inner product defined for $x\in\widetilde{\mathcal{U}}_{j}$, for $j=1,2,\ldots,N$, by
    \begin{equation*}
        \bigl<v,w\bigr>_{j}\coloneqq{}
        |\det(\nabla\sigma_{j}(x))|
        v^{T}\nabla\sigma_{j}\bigl(\sigma_{j}(x)\bigr)
        \nabla\sigma_{j}\bigl(\sigma_{j}(x)\bigr)^{T}w
    \end{equation*}
    where $v,w\in\mathbb{R}^{2}$.
    This inner product will enter when verifying that the PDEs glue properly and to simplify the notation hereafter.
    In addition, using the notation
    \begin{equation*}
        \mathcal{D}(x)\coloneqq
        \begin{cases}
            1& \text{for }x\in\widetilde{U}_{j}\cap\Omega,\\
            \frac{1-(\widetilde{\psi}_{j}^{-1})^{2}(x)\kappa\bigl((\widetilde{\psi}_{j}^{-1})^{1}(x)\bigr)}
            {1+(\widetilde{\psi}_{j}^{-1})^{2}(x)\kappa\bigl((\widetilde{\psi}_{j}^{-1})^{1}(x)\bigr)},
            &\text{for }x\in\widetilde{U}_{j}\setminus\Omega,
        \end{cases}
    \end{equation*}
    previously called the \emph{distortion factor}, we define
    \begin{equation*}
        \text{div}_{j}(w)(x)\coloneqq
        |\det(\nabla\sigma_{j}(x))|^{\frac{1}{2}}
        \biggl[
        \mathcal{D}(x)
        \partial_{\tau}\biggl(w(x)\cdot\tau\bigl((\widetilde{\psi}_{j}^{-1})^{1}(x)\bigr)\biggr)
        +\partial_{\nu}\biggl(w(x)\cdot\nu\bigl((\widetilde{\psi}_{j}^{-1})^{1}(x)\bigr)\biggr)
        \biggr]
    \end{equation*}
    for $x\in\widetilde{\mathcal{U}}_{j}$
    for $j=1,2,\ldots,N$ and functions,
    $w$, of appropriate regularity.
    Note that $\text{div}_{j}$ here is notation and does not match the
    definition of divergence from Riemannian geometry.
    In particular, we use this notation to denote a quantity resembling divergence
    but including a distortion factor.
    We also note that our inner product matches the metric
    used in \cite{DiMiPi}.
    Finally, for notational convenience, we let
    $\mathcal{G}_{j}\colon\widetilde{\Omega}\times{}B_{r_{1}}(0)
    \to{}M_{2\times2}(\mathbb{R})$, for $j=1,2,\ldots,N$,
    denote the matrix-valued functions given by
    \begin{equation*}
        \mathcal{G}_j(x,y)\coloneqq
        \begin{pmatrix}
            \frac{1}{(1-y_{2}\kappa(y_{1}))^{2}}& 0\\
            0& 1
        \end{pmatrix}
        \, ,
    \end{equation*}
    where $\kappa$ denotes the curvature of $\partial\Omega$ at the point $x=\widetilde{\psi}_j(y_1,0)$.
    
Here and in the following sections, a constant $C>0$ is generic and can change value from one line to the next.

\begin{lemma}\label{lem:PDEGluing}
    Suppose $u\in{}W_{T}^{1,2}(\Omega;\mathbb{R}^{2})$ solves
    \eqref{eq:WeakForm} and $U$ is the extension of $u$ defined as in
    \eqref{def:uExtension}.
    Then there is a function $\widetilde{\mathcal{F}}_{j}(x,z,p)$, for j=1,2,\ldots,N, satisfying
    \begin{equation*}
        |\widetilde{\mathcal{F}}_{j}(x,z,p)|\le{}C(\Omega,\ktwo)\bigl[1+|z|+|p|\bigr],
    \end{equation*}
    such that
    \begin{align}
        \sum_{i=1}^{2}\kone\int_{\widetilde{\mathcal{U}}_{j}}\!{}
        \bigl<\nabla{}U^{i},\nabla{}v\bigr>_{j}
        +\ktwo\int_{\widetilde{\mathcal{U}}_{j}}\!{}
        \emph{div}_{j}(U)\emph{div}_{j}(v)
        =&\int_{\widetilde{\mathcal{U}}_{j}\cap\Omega}\!{}
        \frac{U(x)\cdot{}v(x)}{\varepsilon^{2}}(1-|U(x)|^{2})
        \nonumber\\
        &+\int_{\widetilde{\mathcal{U}}_{j}\setminus\Omega}\!{}
        |\det(\nabla\sigma_{j}(x))|
        \frac{\mathfrak{R}_{j}\bigl(x,U(x)\bigr)\cdot{}v(x)}{\varepsilon^{2}}(1-|U(x)|^{2})
        \nonumber\\
        &+\int_{\widetilde{\mathcal{U}}_{j}}\!{}\widetilde{\mathcal{F}}_{j}
        \bigl(x,U(x),\nabla{}U(x)\bigr)\cdot{}v(x),
        \label{def:ExtendedPDE}
    \end{align}
    for $v\in{}W_{0}^{1,2}(\widetilde{\mathcal{U}}_{j};\mathbb{R}^{2})$ and
    $j=1,2,\ldots,N$.
    If $\Omega$ has $C^{3,1}-$boundary then $F_{j}$ also satisfies
    \begin{equation*}
        |\nabla_{z,p}\widetilde{\mathcal{F}}_{j}(x,z,p)|\le{}C(\Omega,\ktwo)
    \end{equation*}
    for all $j=1,2,\ldots,N$.
\end{lemma}

\begin{proof}
    Since $u$ satisfies \eqref{eq:WeakForm} then for $v\in{}W_{T}^{1,2}(\Omega;\mathbb{R}^{2})$
    satisfying $\text{supp}(v)\subseteq{}\overline{\mathcal{U}}_{j}$ we can rewrite \eqref{eq:WeakForm}, after expressing both $u$ and $v$ in tangent-normal coordinates, as follows
    \begin{align}
        &\kone\int_{\mathcal{U}_{j}}\!{}
        \biggl[
        \frac{\partial_{y_{1}}\widetilde{u}_{\tau}\bigl(\widetilde{\psi}_{j}^{-1}(x)\bigr)
        \partial_{y_{1}}\widetilde{v}_{\tau}\bigl(\widetilde{\psi}_{j}^{-1}(x)\bigr)}
        {\bigl(1-(\widetilde{\psi}_{j}^{-1})^{2}(x)\kappa\bigl((\widetilde{\psi}_{j}^{-1})^{1}(x)\bigr)\bigr)^{2}}
        +
        \partial_{y_{2}}\widetilde{u}_{\tau}\bigl(\widetilde{\psi}_{j}^{-1}(x)\bigr)
        \partial_{y_{2}}\widetilde{v}_{\tau}\bigl(\widetilde{\psi}_{j}^{-1}(x)\bigr)
        \biggr]
        \nonumber\\
        &+\ktwo\int_{\mathcal{U}_{j}}
        \biggl[
        \frac{
        \partial_{y_{1}}\widetilde{u}_{\tau}\bigl(\widetilde{\psi}_{j}^{-1}(x)\bigr)
        }{1-(\widetilde{\psi}_{j}^{-1})^{2}(x)
        \kappa\bigl((\widetilde{\psi}_{j}^{-1})^{1}(x)\bigr)}
        +\partial_{y_{2}}\widetilde{u}_{\nu}\bigl(\widetilde{\psi}_{j}^{-1}(x)\bigr)
        \biggr]
        \biggl[\frac{
        \partial_{y_{1}}\widetilde{v}_{\tau}\bigl(\widetilde{\psi}_{j}^{-1}(x)\bigr)
        }{1-(\widetilde{\psi}_{j}^{-1})^{2}(x)
        \kappa\bigl((\widetilde{\psi}_{j}^{-1})^{1}(x)\bigr)}
        \biggr]
        \nonumber\\
        =&\int_{\mathcal{U}_{j}}\!{}
        \frac{\widetilde{u}_{\tau}\bigl(\widetilde{\psi}_{j}^{-1}(x)\bigr)
        \widetilde{v}_{\tau}\bigl(\widetilde{\psi}_{j}^{-1}(x)\bigr)}
        {\varepsilon^{2}}(1-|u|^{2})
        +\int_{\mathcal{U}_{j}}\!{}F_{j,\tau}\bigl(x,u(x),\nabla{}u(x)\bigr)
        \widetilde{v}_{\tau}\bigl(\widetilde{\psi}_{j}^{-1}(x)\bigr)
        \label{eq:TangentialPDE}
    \end{align}
    and
    \begin{align}
        &\kone\int_{\mathcal{U}_{j}}\!{}
       \biggl[
        \frac{\partial_{y_{1}}\widetilde{u}_{\nu}\bigl(\psi_{j}^{-1}(x)\bigr)
        \partial_{y_{1}}\widetilde{v}_{\nu}\bigl(\psi_{j}^{-1}(x)\bigr)}
        {\bigl(1-(\psi_{j}^{-1})^{2}(x)\kappa\bigl((\psi_{j}^{-1})^{1}(x)\bigr)\bigr)^{2}}
        +
        \partial_{y_{2}}\widetilde{u}_{\nu}\bigl(\psi_{j}^{-1}(x)\bigr)
        \partial_{y_{2}}\widetilde{v}_{\nu}\bigl(\psi_{j}^{-1}(x)\bigr)
        \biggr]\nonumber\\
        &+\ktwo\int_{\mathcal{U}_{j}}
        \biggl[\frac{
        \partial_{y_{1}}\widetilde{u}_{\tau}\bigl(\psi_{j}^{-1}(x)\bigr)
        }{1-(\widetilde{\psi}_{j}^{-1})^{2}(x)
        \kappa\bigl((\widetilde{\psi}_{j}^{-1})^{1}(x)\bigr)}
        +\partial_{y_{2}}\widetilde{u}_{\nu}\bigl(\psi_{j}^{-1}(x)\bigr)
        \biggr]
        \Bigl[\partial_{y_{2}}\widetilde{v}_{\nu}\bigl(\psi_{j}^{-1}(x)\bigr)
        \Bigr]\nonumber\\
        =&\int_{\mathcal{U}_{j}}\!{}
        \frac{\widetilde{u}_{\nu}\bigl(\psi_{j}^{-1}(x)\bigr)
        \widetilde{v}_{\nu}\bigl(\psi_{j}^{-1}(x)\bigr)}
        {\varepsilon^{2}}(1-|u|^{2})
        +\int_{\mathcal{U}_{j}}\!{}F_{j,\nu}\bigl(x,u(x),\nabla{}u(x)\bigr)
        \widetilde{v}_{\nu}\bigl(\psi_{j}^{-1}(x)\bigr)
        \, ,
        \label{eq:NormalPDE}
    \end{align}
    where $F_{j,\tau}$ and $F_{j,\nu}$ are determined by the remaining integrands found in the detailed calculations presented in Appendix \ref{app:calculations}.
    From there it becomes clear that $F_{j,\tau}$ and $F_{j,\nu}$ satisfy the grwoth estimate
    \begin{equation}\label{eq:carathmap}
        \max\{|F_{j,\tau}(x,z,p)|,|F_{j,\nu}(x,z,p)|\}
        \le{}C(\Omega,\ktwo)\bigl[1+|z|+|p|\bigr]
    \end{equation}
    for each $j=1,2\ldots,N$.
    In addition, if $\Omega$ has $C^{3,1}-$boundary then we also have
    \begin{equation}\label{eq:carathmapgrad}
        \max\{|\nabla_{z,p}F_{j,\tau}(x,z,p)|,|\nabla_{z,p}F_{j,\nu}(x,z,p)|\}
        \le{}C(\Omega,\ktwo)
    \end{equation}
    for each $j=1,2,\ldots,N$.\\

    Next, we use \eqref{eq:TangentialPDE} and \eqref{eq:NormalPDE}
    in order to show that the extension $U$
    satisfies \eqref{def:ExtendedPDE} on $\widetilde{\mathcal{U}}_{j}$ for $j=0,1,2\ldots,N$.
    Since $U=u$ on $\widetilde{\mathcal{U}}_{0}$ then we may restrict attention to
    the case of $j=1,2,\ldots,N$ for $x\in\widetilde{\mathcal{U}}_{j}$ for $j=1,2,\ldots,N$.
    Through direct computation using the definition of $U$ in
    \eqref{def:uExtension} and the PDEs satisfied by its components \eqref{eq:TangentialPDE}, \eqref{eq:NormalPDE},
    as well as the Change of Variables Theorem we obtain
    \begin{align*}
        &
        \sum_{i=1}^{2}\kone\int_{\widetilde{\mathcal{U}}_{j}}\!{}
        \bigl<\nabla{}U^{i}(x),\nabla{}v^{i}(x)\bigr>_{j}
        +\ktwo\int_{\widetilde{\mathcal{U}}_{j}}\!{}
        \text{div}_{j}(U)(x)\text{div}_{j}(v)(x)
        \\
        =&\kone\int_{\widetilde{\mathcal{U}}_{j}\cap\Omega}\!{}
        \nabla\widetilde{u}_{\tau}\bigl(\widetilde{\psi}_{j}^{-1}(x)\bigr)^{T}
        \bigl[\mathcal{G}_{j}(x,\widetilde{\psi}_{j}^{-1}(x))\bigr]
        \bigl[\nabla\widetilde{v}_{\tau}\bigl(\widetilde{\psi}_{j}^{-1}(x)\bigr)+
        \nabla\widetilde{v}_{\tau}^{R}\bigl(\widetilde{\psi}_{j}^{-1}(x)\bigr)\bigr]\\
        &+\kone
        \int_{\widetilde{\mathcal{U}}_{j}\cap\Omega}\!{}
        \nabla\widetilde{u}_{\nu}\bigl(\widetilde{\psi}_{j}^{-1}(x)\bigr)^{T}
        \bigl[\mathcal{G}_{j}(x,\widetilde{\psi}_{j}^{-1}(x))\bigr]
        \bigl[\nabla\widetilde{v}_{\nu}\bigl(\widetilde{\psi}_{j}^{-1}(x)\bigr)-
        \nabla\widetilde{v}_{\nu}^{R}\bigl(\widetilde{\psi}_{j}^{-1}(x)\bigr)\bigr]\\
        &+\ktwo\int_{\widetilde{\mathcal{U}}_{j}\cap\Omega}\!{}
        \biggl[
        \frac{\partial_{y_{1}}\widetilde{u}_{\tau}\bigl(\widetilde{\psi}_{j}^{-1}(x)\bigr)}
        {1-(\widetilde{\psi}_{j}^{-1})^{2}(x)\kappa\bigl((\widetilde{\psi}_{j}^{-1})^{1}(x)\bigr)}
        +\partial_{y_{2}}\widetilde{u}_{\nu}\bigl(\widetilde{\psi}_{j}^{-1}(x)\bigr)
        \biggr]
        \frac{\partial_{y_{1}}\widetilde{v}_{\tau}\bigl(\widetilde{\psi}_{j}^{-1}(x)\bigr)+\partial_{y_{1}}(\widetilde{v}^{R})_{\tau}\bigl(\widetilde{\psi}_{j}^{-1}(x)\bigr)}
        {1-(\widetilde{\psi}_{j}^{-1})^{2}(x)\kappa\bigl((\widetilde{\psi}_{j}^{-1})^{1}(x)\bigr)}
        \\
        &+\ktwo\int_{\widetilde{\mathcal{U}}_{j}\cap\Omega}\!{}
        \biggl[
        \frac{\partial_{y_{1}}\widetilde{u}_{\tau}\bigl(\widetilde{\psi}_{j}^{-1}(x)\bigr)}
        {1-(\widetilde{\psi}_{j}^{-1})^{2}(x)\kappa\bigl((\widetilde{\psi}_{j}^{-1})^{1}(x)\bigr)}
        +\partial_{y_{2}}\widetilde{u}_{\nu}\bigl(\widetilde{\psi}_{j}^{-1}(x)\bigr)
        \biggr]
        \bigl[\partial_{y_{2}}\widetilde{v}_{\nu}\bigl(\widetilde{\psi}_{j}^{-1}(x)\bigr)-\partial_{y_{2}}(\widetilde{v}^{R})_{\nu}\bigl(\widetilde{\psi}_{j}^{-1}(x)\bigr)
        \bigr]\\
        &+\int_{\widetilde{\mathcal{U}}_{j}\setminus\Omega}\!{}
        \widetilde{F}_{j}\bigl(x,U(x),\nabla{}U(x)\bigr)
        v(x).
    \end{align*}
    where $v^{R}(y)\coloneqq{}v(Ry)$ and where $\widetilde{F}_{j}$
    satisfies \eqref{eq:carathmap} and, if $\Omega$ has $C^{3,1}-$boundary,
    \eqref{eq:carathmapgrad}.
    Again we refer to Appendix \ref{app:calculations} for the details.
    We introduce the even part of the function $v$
    \begin{equation*}
        v_{E}(x)\coloneqq{}\frac{v(x)+\mathfrak{R}_{j}\bigl(x,v(\sigma_{j}(x))\bigr)}{2}.
    \end{equation*}
    This can be expressed as
    \begin{equation*}
        v_{E}(x)=
        \frac{\widetilde{v}_{\tau}\bigl(\widetilde{\psi}_{j}^{-1}(x)\bigr)
        +\widetilde{v}_{\tau}\bigl(R\widetilde{\psi}_{j}^{-1}(x)\bigr)}{2}
        \tau\bigl((\widetilde{\psi}_{j}^{-1})^{1}(x)\bigr)
        +\frac{\widetilde{v}_{\nu}\bigl(\widetilde{\psi}_{j}^{-1}(x)\bigr)
        -\widetilde{v}_{\nu}\bigl(R\widetilde{\psi}_{j}^{-1}(x)\bigr)}{2}
        \nu\bigl((\widetilde{\psi}_{j}^{-1})^{1}(x)\bigr)
    \end{equation*}
    and hence for $x\in\partial\Omega$ we have
    \begin{equation*}
        (v_{E})_{\nu}(x)=
        \frac{\widetilde{v}_{\nu}\bigl(\widetilde{\psi}_{j}^{-1}(x)\bigr)
        -\widetilde{v}_{\nu}\bigl(R\widetilde{\psi}_{j}^{-1}(x)\bigr)}{2}
        =\frac{\widetilde{v}_{\nu}(y_{1},0)-\widetilde{v}_{\nu}(y_{1},0)}{2}=0.
    \end{equation*}
    Because of the above, we can use
    \eqref{eq:TangentialPDE}, \eqref{eq:NormalPDE}, and that $U=u$ on
    $\Omega$ to obtain
    \begin{equation*}
        \sum_{i=1}^{2}\kone\int_{\widetilde{\mathcal{U}}_{j}}\!{}
        \bigl<\nabla{}U^{i},\nabla{}v^{i}\bigr>_{j}
        +\ktwo\int_{\widetilde{\mathcal{U}}_{j}}\!{}
        \text{div}_{j}(U)\text{div}_{j}(v)
        =2\int_{\widetilde{\mathcal{U}}_{j}\cap\Omega}\!{}
        \frac{U\cdot{}v_{E}}{\varepsilon^{2}}(1-|U|^{2})
        +\int_{\widetilde{\mathcal{U}}_{j}}\!{}
        \widetilde{\mathcal{F}}_{j}\bigl(x,U(x),\nabla{}U(x)\bigr)\cdot
        v(x),
    \end{equation*}
    where $\widetilde{\mathcal{F}}_j$ depends on $\widetilde{F}_j, F_{j,\nu}$ and $F_{j,\tau}$.
    Noting that
    \begin{equation*}
        U(x)\cdot\bigl[\mathfrak{R}_{j}\bigl(x,v(\sigma_{j}(x))\bigr)\bigr]=\mathfrak{R}_{j}(x,U(x))\cdot{}v(\sigma_{j}(x))
        \, ,
    \end{equation*}
    and changing variables gives \eqref{def:ExtendedPDE}.
\end{proof}

Next, we show that the system of PDEs solved by the tangential and normal components of the coordinate representations of $U$ satisfies the Legendre-Hadamard ellipticicty condition.
This is sufficient since representing $u$ and $v$ in a new basis results, to highest order, in
\begin{equation*}
    \eta^{i}=\sum_{m=1,2}C_{m}^{i}\widetilde{\eta}^{m},\hspace{20pt}
    \eta^{j}=\sum_{n=1,2}C_{n}^{j}\widetilde{\eta}^{n}
\end{equation*}
which gives
\begin{align*}
    \sum_{\alpha,\beta,i,j=1,2}A_{i,j}^{\alpha,\beta}\xi_{\alpha}\xi_{\beta}\eta^{i}\eta^{j}
    &=\sum_{\alpha,\beta,i,j=1,2}\sum_{m=1,2}\sum_{n=1,2}
    A_{i,j}^{\alpha,\beta}C_{m}^{i}C_{n}^{j}\xi_{\alpha}\xi_{\beta}
    \widetilde{\eta}^{m}\widetilde{\eta}^{n}\\
    &=\sum_{\alpha,\beta,m,n=1,2}
    \Bigl[\sum_{i,j=1,2}A_{i,j}^{\alpha,\beta}C_{m}^{i}C_{n}^{j}\Bigr]
    \xi_{\alpha}\xi_{\beta}\widetilde{\eta}^{m}\widetilde{\eta}^{n}
\end{align*}
and hence one representation will be non-negative if and only if the other is.
\begin{lemma}\label{lem:ExtensionEllipticity}
    Suppose $u\in{}W_{T}^{1,2}(\Omega;\mathbb{R}^{2})$ solves
    \eqref{eq:WeakForm} and suppose that $U$ is the extension of $u$ defined
    as in \eqref{def:uExtension}.
    Then $U$ weakly solves an elliptic PDE.
\end{lemma}

\begin{proof}
    Since we have already demonstrated that $U$ solved a PDE in Lemma
    \ref{lem:PDEGluing} then we focus on demonstrating that this PDE is
    elliptic.
    In addition, we focus on $x\in\widetilde{\mathcal{U}}_{j}\setminus\Omega$
    for $j=1,2,\ldots,N$ since ellipticty for $x\in\Omega$ is established through a
    similar proof.
    In order to demonstrate ellipticity we first rewrite the left-hand side of
    \eqref{def:ExtendedPDE} in terms of tangential and normal coordinates as
    well as derivatives in the $y$ variables.
    Doing this we obtain
    \begin{align*}
        \sum_{i=1}^{2}\kone\bigl<\nabla{}U^{i}(x),\nabla{}v^{i}(x)\bigr>_{j}&+
        \ktwo\text{div}_{j}(U)\text{div}_{j}(v) \\
        =&|\det(\nabla\sigma_{j}(x))|\biggl[\frac{(1+\ktwo)\partial_{y_{1}}
        \widetilde{U}_{\tau}(\widetilde{\psi}_{j}^{-1}(x))
        \partial_{y_{1}}\widetilde{v}_{\tau}(\widetilde{\psi}_{j}^{-1}(x))}
        {(1+(\widetilde{\psi}_{j}^{-1})^{2}\kappa((\widetilde{\psi}_{j}^{-1})^{1}(x)))^{2}}\\
        &+\frac{\ktwo\partial_{y_{1}}
        \widetilde{U}_{\tau}(\widetilde{\psi}_{j}^{-1}(x))
        \partial_{y_{2}}\widetilde{v}_{\nu}(\widetilde{\psi}_{j}^{-1}(x))}
        {1+(\widetilde{\psi}_{j}^{-1})^{2}\kappa((\widetilde{\psi}_{j}^{-1})^{1}(x))}
        +\kone\partial_{y_{2}}\widetilde{U}_{\tau}(\widetilde{\psi}_{j}^{-1}(x))
        \partial_{y_{2}}\widetilde{v}_{\tau}(\widetilde{\psi}_{j}^{-1}(x))\\
        &+\frac{\kone\partial_{y_{1}}
        \widetilde{U}_{\nu}(\widetilde{\psi}_{j}^{-1}(x))
        \partial_{y_{1}}\widetilde{v}_{\nu}(\widetilde{\psi}_{j}^{-1}(x))}
        {(1+(\widetilde{\psi}_{j}^{-1})^{2}(x)\kappa((\widetilde{\psi}_{j}^{-1})^{1}(x)))^{2}}
        +\frac{\ktwo\partial_{y_{2}}
        \widetilde{U}_{\nu}(\widetilde{\psi}_{j}^{-1}(x))
        \partial_{y_{1}}\widetilde{v}_{\tau}(\widetilde{\psi}_{j}^{-1}(x))}
        {1+(\widetilde{\psi}_{j}^{-1})^{2}\kappa((\widetilde{\psi}_{j}^{-1})^{1}(x))}\\
        &+(1+\ktwo)\partial_{y_{2}}\widetilde{U}_{\nu}(\widetilde{\psi}_{j}^{-1}(x))
        \partial_{y_{2}}\widetilde{v}_{\nu}(\widetilde{\psi}_{j}^{-1}(x))\biggr]
        +\text{lower order terms}.
    \end{align*}
    Dividing by $|\det(\nabla\sigma_{j}(x))|$ and
    integrating over $\widetilde{\mathcal{U}}_{j}$ and changing coordinates gives a
    leading order of
    \begin{align*}
        \int_{\widetilde{\mathcal{U}}_{j}}\!{}
        \frac{1}{|\det(\nabla\sigma_{j}(x))|}
        \bigl[\kone&\bigl<\nabla{}U^{i}(x),\nabla{}v^{i}(x)\bigr>_{j}+
        \ktwo\text{div}_{j}(U)\text{div}_{j}(v)\bigr] \\
        =&\int_{B_{r_{1}}(0)}\!{}\frac{(1+\ktwo)\partial_{y_{1}}
        \widetilde{U}_{\tau}(y)
        \partial_{y_{1}}\widetilde{v}_{\tau}(y)}
        {(1+y_{2}\kappa(y_{1}))^{2}}
        +\int_{B_{r_{1}}(0)}\!{}\frac{\ktwo\partial_{y_{1}}
        \widetilde{U}_{\tau}(y)
        \partial_{y_{2}}\widetilde{v}_{\nu}(y)}
        {1+y_{2}\kappa(y_{1})}\\
        &+\int_{B_{r_{1}}(0)}\!{}\kone\partial_{y_{2}}\widetilde{U}_{\tau}(y)
        \partial_{y_{2}}\widetilde{v}_{\tau}(y)
        +\int_{B_{r_{1}}(0)}\!{}\frac{\kone\partial_{y_{1}}
        \widetilde{U}_{\nu}(y)
        \partial_{y_{1}}\widetilde{v}_{\nu}(y)}
        {(1+y_{2}\kappa(y_{1}))^{2}}\\
        &+\int_{B_{r_{1}}(0)}\!{}\frac{\ktwo\partial_{y_{2}}
        \widetilde{U}_{\nu}(y)
        \partial_{y_{1}}\widetilde{v}_{\tau}(y)}
        {1+y_{2}\kappa(y_{1})}
        +\int_{B_{r_{1}}(0)}\!{}(1+\ktwo)\partial_{y_{2}}\widetilde{U}_{\nu}(y)
        \partial_{y_{2}}\widetilde{v}_{\nu}(y)\\
        &+\text{lower order terms}.
    \end{align*}
    Note that we omit the lower order terms since ellipticity is determined by the highest order.
    We see that the tangential and normal coordinate components of $U$ solve a PDE system
    for which
    \begin{align*}
        A_{1,1}^{1,1}&=\frac{1+\ktwo}{(1+y_{2}\kappa(y_{1}))^{2}},\hspace{5pt}
        A_{2,2}^{1,1}=\frac{1}{(1+y_{2}\kappa(y_{1}))^{2}},\hspace{5pt}
        A_{2,1}^{1,2}=\frac{\ktwo}{1+y_{2}\kappa(y_{1})},\\
        A_{1,2}^{2,1}&=\frac{\ktwo}{1+y_{2}\kappa(y_{1})},\hspace{12pt}
        A_{1,1}^{2,2}=1,\hspace{65pt}
        A_{2,2}^{2,2}=1+\ktwo
    \end{align*}
    and all other coefficients are zero.
    We refer to Definition 3.36 in Subsection 3.41 of \cite{GiMa} for notation regarding PDE systems.
    Notice that for $\xi,\eta\in\mathbb{R}^{2}$ we have
    \begin{align*}
        \sum_{\alpha,\beta,i,j=1,2}A_{i,j}^{\alpha,\beta}
        \xi_{\alpha}\xi_{\beta}\eta^{i}\eta^{j}
        &=\frac{1+\ktwo}
        {(1+y_{2}\kappa(y_{1}))^{2}}\xi_{1}\xi_{1}\eta^{1}\eta^{1}
        +\frac{1}{(1+y_{2}\kappa(y_{1}))^{2}}\xi_{1}\xi_{1}\eta^{2}\eta^{2}\\
        &+\frac{2\ktwo}{1+y_{2}\kappa(y_{1})}\xi_{1}\xi_{2}\eta^{2}\eta^{1}
        +\kone\xi_{2}\xi_{2}\eta^{1}\eta^{1}
        +(1+\ktwo)\xi_{2}\xi_{2}\eta^{2}\eta^{2}.
    \end{align*}
    By the Arithmetic-Geometric inequality we have
    \begin{align*}
        \frac{2\ktwo}{1+y_{2}\kappa(y_{1})}
        \xi_{1}\xi_{2}\eta^{2}\eta^{1}
        &=2\biggl[\frac{\sqrt{\ktwo}}{1+y_{2}\kappa(y_{1})}\xi_{1}\eta^{1}\biggr]\bigl[\sqrt{\ktwo}\xi_{2}\eta^{2}\bigr]\\
        &\le\frac{\ktwo}
        {(1+y_{2}\kappa(y_{1}))^{2}}\xi_{1}\xi_{1}\eta^{1}\eta^{1}
        +\ktwo\xi_{2}\xi_{2}\eta^{2}\eta^{2}
    \end{align*}
    and hence we obtain
    \begin{equation*}
        \sum_{\alpha,\beta,i,j=1,2}A_{i,j}^{\alpha,\beta}\xi_{\alpha}\xi_{\beta}\eta^{i}\eta^{j}
        \ge\frac{\kone\xi_{1}^{2}|\eta|^{2}}{(1+y_{2}\kappa(y_{1}))^{2}}
        +\kone\xi_{2}^{2}|\eta|^{2}
        \ge\kone\min\biggl\{\frac{1}{(1+y_{2}\kappa(y_{1}))^{2}},1\biggr\}|\xi|^{2}|\eta|^{2}
        \, ,
    \end{equation*}
    which demonstrates Legendre-Hadamard ellipticity.
\end{proof}

\section{Uniform \texorpdfstring{$L^{4}$}{} Estimate}
\label{sec:Regularity}

In this section we establish preliminary $L^{4}$ bounds, uniform in $\varepsilon$, similar to those found in \cite{CoKePh}.
We use the estimates to establish the desired regularity in Section~\ref{sec:Rescaling}.\\

Here we first establish a uniform $L^{1}$ bound on 
$\frac{1}{\varepsilon^{2}}\bigl(|u_{\varepsilon}|^{2}-1\bigr)^{2}$
using results for minimizers of the Ginzburg-Landau energy.
We will also need the following lemma concerning subsequences.
\begin{lemma}\label{lem:SubseqBound}
    Suppose $A\in\mathbb{R}$, $a>0$, and $f\colon(0,a]\to\mathbb{R}$
    is such that for any sequence $\{x_{n}\}_{n\in\mathbb{N}}\subseteq(0,a]$
    where $\lim\limits_{n\to\infty}x_{n}=0$ there
    is a subsequence $\{x_{n_{m}}\}_{m\in\mathbb{N}}$ such that
    \begin{equation*}
        f(x_{n_{m}})\ge{}A.
    \end{equation*}
    Then there is $\delta>0$ such that if $x\in(0,\delta)$ then
    \begin{equation*}
        f(x)\ge{}A.
    \end{equation*}
\end{lemma}

\begin{proof}
    Suppose that no such $\delta>0$ exists.
    Then there is a sequence $\{x_{n}\}_{n\in\mathbb{N}}\subseteq(0,a]$
    tending to zero so that $f(x_{n})<A$.
    By passing to a further subsequence, $\{x_{n_{m}}\}_{m\in\mathbb{N}}$,
    we find that $A\le{}f(x_{n_{m}})<A$ which is a contradiction.
\end{proof}

Now we show that we may obtain uniform $O(1)$ bounds similar to Corollary $2.1$
of \cite{CoKePh}.
We do this by following the same procedure as in Proposition $2.1$ of
\cite{CoKePh}.
\begin{lemma}\label{lem:Order1Potential}
    Suppose $u_{\varepsilon}\in{}W_{T}^{1,2}(\Omega;\mathbb{R}^{2})$
    is a minimizer of \eqref{def:Energy} among functions in
    $W_{T}^{1,2}(\Omega;\mathbb{R}^{2})$.
    Then there exists a constant $C(\Omega,\ktwo)$ and
    $\varepsilon_{0}>0$, dependent only on $\Omega$ and $\ktwo$,
    such that for all $\varepsilon\in(0,\varepsilon_{0})$
    \begin{equation*}
        \int_{\Omega}\!{}
        \biggl\{\frac{\ktwo}{2}\bigl(\emph{div}(u_{\varepsilon})\bigr)^{2}
        +\frac{1}{8\varepsilon^{2}}\bigl(|u_{\varepsilon}|^{2}-1\bigr)^{2}\biggr\}
        \le{}C(\Omega,\ktwo).
    \end{equation*}
\end{lemma}

\begin{proof} 
The proof closely follows the approach used in Proposition $2.1$ of
\cite{CoKePh}.
Specifically, it proceeds in two steps:
\begin{enumerate}
    \item
        We obtain an upper bound for \eqref{def:Energy} by direct construction of
        an approximation to the ideal vortex configuration.
        This is done by using a cutoff function as a radial profile as well
        as carefully choosing the vortex to be divergence free.
    \item
        We show that the upper bound can be saturated, to logarithmic order,
        by appealing to lower bound results for the Ginzburg-Landau energy.
\end{enumerate}
    \hspace{2pt}\\
    \noindent\emph{Step 1: Upper Bound.}\hspace{2pt}\\\par{} Suppose $x_{0}\in\Omega$ and choose $r_{1}$ so that
$r_{1}\in\bigl(0,\text{dist}(x_{0},\partial\Omega)\bigr)$.
For $\varepsilon\in(0,r_{1})$ we define
$U_{\varepsilon}(x)\coloneqq\rho_{\varepsilon}(x)\frac{(x-x_{0})^{\perp}}{|x-x_{0}|}$ 
on $\overline{B_{r_{1}}(x_{0})}$
where
$\rho_{\varepsilon}(x)\coloneqq
\min\bigl\{\frac{|x-x_{0}|}{\varepsilon},1\bigr\}$.
We also set $U_{\varepsilon}(x)=\tau(x)$ on $\partial\Omega$.
As in the discussion above Theorem $I.3$ from \cite{BBH} or
above equation 2.7 of \cite{CoKePh} we may find a function
$v\in{}W^{1,2}(\Omega\setminus{}B_{r_{1}}(x_{0});\mathbb{S}^{1})$ which
extends the boundary values of $u_{\varepsilon}$ on
$\partial{}B_{r_{1}}(x_{0})$ and $\partial\Omega$ since
$\text{deg}(\partial\Omega)=\text{deg}(\partial{}B_{r_{1}}(x_{0}))$.
We now define $U_{\varepsilon}$ on $\overline{\Omega}$ by
\begin{equation*}
    U_{\varepsilon}(x)=
    \begin{cases}
        \rho_{\varepsilon}(x)\frac{(x-x_{0})^{\perp}}{|x-x_{0}|}&
        \text{on }B_{r_{1}}(x_{0}),\\
        v(x)&\text{on }\Omega\setminus{}B_{r_{1}}(x_{0}),\\
        \tau(x)&\text{on }\partial\Omega.
    \end{cases}
\end{equation*}
We notice that by Theorem $I.3$ of \cite{BBH} we have
\begin{align*}
    \int_{\Omega\setminus{}B_{r_{1}}(x_{0})}\!{}
    \biggl\{\frac{1}{2}|\nabla{}U_{\varepsilon}|^{2}
    +\frac{\ktwo}{2}\bigl(\text{div}(U_{\varepsilon})\bigr)^{2}
    +\frac{1}{4\varepsilon^{2}}\bigl(|U_{\varepsilon}|^{2}-1\bigr)^{2}\biggr\}
    &\le\biggl(\frac{1}{2}+\ktwo\biggr)
    \int_{\Omega\setminus{}B_{r_{1}}(x_{0})}\!{}
    |\nabla{}v|^{2}\\
    &\le\biggl(\frac{1}{2}+\ktwo\biggr)
    \int_{\Omega\setminus{}B_{r_{1}}(x_{0})}\!{}|\nabla\Phi_{2}|^{2}
    +C(\Omega)
\end{align*}
where $\Phi_{2}\colon\Omega\setminus{}B_{r_{1}}(x_{0})\to\mathbb{R}$ satisfies 
\begin{equation*}
    \begin{cases}
        -\Delta\Phi_{2}=0&\text{on }\Omega\setminus{}B_{r_{1}}(x_{0}),\\
        \partial_{\nu}\Phi_{2}(x)=\kappa&\text{on }\partial\Omega,\\
        \partial_{\nu}\Phi_{2}(x)=\frac{1}{r_{1}}
        &\text{on }\partial{}B_{r_{1}}(x_{0}).
    \end{cases}
\end{equation*}
Since $\Phi_{2}$ only depends on quantities dependent on $\Omega$ then,
by Theorem $2.4.2.6$ of \cite{Gr}, we conclude that
\begin{equation*}
    \int_{\Omega\setminus{}B_{r_{1}}(x_{0})}\!{}
    \biggl\{\frac{1}{2}|\nabla{}U_{\varepsilon}|^{2}
    +\frac{\ktwo}{2}(\text{div}(U_{\varepsilon}))^{2}
    +\frac{1}{4\varepsilon^{2}}\bigl(|U_{\varepsilon}|^{2}-1\bigr)^{2}\biggr\}
    \le{}C(\Omega,\ktwo).
\end{equation*}
We observe that $\text{div}\bigl(\frac{(x-x_{0})^{\perp}}{|x-x_{0}|}\bigr)=0$
and hence
\begin{equation*}
    \int_{B_{r_{1}}(x_{0})}\!{}
    \biggl\{\frac{1}{2}|\nabla{}U_{\varepsilon}|^{2}
    +\frac{\ktwo}{2}\bigl(\text{div}(U_{\varepsilon})\bigr)^{2}
    +\frac{1}{4\varepsilon^{2}}\bigl(|U_{\varepsilon}|^{2}-1\bigr)^{2}\biggr\}
    \le{}\pi\kone\mathopen{}\left|\log(\varepsilon)\right|\mathclose{}
    +C(\Omega,\ktwo).
\end{equation*}
Altogether we have, since $u_{\varepsilon}$ is a minimizer for
$\mathscr{F}_{\varepsilon}$, that
\begin{equation}\label{eq:DivEnergyUpperBound}
    GL_{\varepsilon,\text{div}}(u_{\varepsilon})
    \le{}GL_{\varepsilon,\text{div}}(U_{\varepsilon})
    \le{}\pi\kone\mathopen{}\left|\log(\varepsilon)\right|\mathclose{}
    +C(\Omega,\ktwo)
\end{equation}
for $\varepsilon\in(0,r_{1})$.\\[5pt]

\noindent\emph{Step 2: Lower Bound.}\hspace{2pt}\\
\par{} We now obtain a lower bound for
\begin{equation*}
    \inf_{u\in{}W_{T}^{1,2}(\Omega;\mathbb{R}^{2})}
    \Bigl\{GL_{\varepsilon}(u,\Omega)\Bigr\}
\end{equation*}
which matches \eqref{eq:DivEnergyUpperBound} to highest order.
For each $\varepsilon\in\bigl(0,r_{1}\bigr]$,
let $w_{\varepsilon}$
be the minimizer of $GL_{\varepsilon}$ (i.e. of the
Ginzburg-Landau energy) over
$W_{T}^{1,2}(\Omega;\mathbb{R}^{2})$.
Let $\{\varepsilon_{\nu}\}_{\nu=1}^{\infty}\subseteq(0,r_{1}]$ be
any sequence tending to zero.
By Lemma $5.3$ and equation $(5.4)$ of \cite{ABB} we obtain,
up to passing to a subsequence, $\{\varepsilon_{\nu_{l}}\}_{l=1}^{\infty}$,
that
\begin{align*}
    GL_{\varepsilon_{\nu_{l}}}(w_{\varepsilon_{\nu_{l}}},\Omega)
    \ge\biggl(\pi\sum_{i=1}^{I}|d_{x_{i}}|
    +\frac{\pi}{2}\sum_{j=1}^{J}|d_{y_{j}}|\biggr)
    \mathopen{}\left|\log(\varepsilon_{\nu_{l}})\right|\mathclose{}-C(\Omega)
\end{align*}
for all $l\in\mathbb{N}$ and where $d_{x_{i}},d_{y_{j}}\neq0$ for all
$i=1,2,\ldots,I$ and $j=1,2,\ldots,J$ in the sum.
Note that while \cite{ABB} assumes that $\partial\Omega$ is $C^{4,\alpha}$ for $\alpha>0$
the proof of the lower bound only requires $C^{2,1}-$regularity.
By \eqref{eq:DivEnergyUpperBound} and Proposition $4.1$ of \cite{ABB} we
conclude that, up to passing to a further
subsequence $\{\varepsilon_{\nu_{l}}\}_{l=1}^{\infty}$,
we must either have one interior defect (i.e.\ $I=1$, $J=0$, and $|d_{x_{1}}|=1$) or two boundary defects ($I=0$, $J=2$, and $|d_{y_{j}}|=1$ for $j=1,2$).
In either case, we obtain
\begin{align*}
    GL_{\varepsilon_{\nu_{l}}}(w_{\varepsilon_{\nu_{l}}},\Omega)
    \ge\pi\mathopen{}\left|\log(\varepsilon_{\nu_{l}})\right|\mathclose{}-C(\Omega)
\end{align*}
for all $l\in\mathbb{N}$.
Since the sequence we started with was arbitrary, we conclude
from Lemma \ref{lem:SubseqBound} that
this inequality holds for all $\varepsilon\in(0,r_{2})$
where $r_{2}>0$ is chosen sufficiently small.
Thus,
\begin{equation}\label{eq:DivEnergyLowerBound}
    GL_{\varepsilon}(w_{\varepsilon},\Omega)
    -\kone\pi\mathopen{}\left|\log(\varepsilon)\right|\mathclose{}
    \ge{}C(\Omega)
\end{equation}
for all $\varepsilon\in(0,r_{2})$.
Following the technique from Corollary $2.1$ of \cite{CoKePh}
we combine \eqref{eq:DivEnergyUpperBound} and \eqref{eq:DivEnergyLowerBound}
to find that for all $\varepsilon\in(0,r_{2}\slash2)$
\begin{align*}
    \int_{\Omega}\!{}
    \biggl\{\frac{\ktwo}{2}(\text{div}(u_{\varepsilon}))^{2}
    +\frac{1}{8\varepsilon^{2}}(|u_{\varepsilon}|^{2}-1)^{2}\biggr\}
    &= GL_{\varepsilon,\text{div}}(u_{\varepsilon},\Omega)
    -GL_{\varepsilon\sqrt{2}}(u_{\varepsilon},\Omega)\\
    &\le{}C(\Omega,\ktwo)
    -
    \Bigl[GL_{\varepsilon\sqrt{2}}(u_{\varepsilon\sqrt{2}},\Omega)
    -\pi\kone\mathopen{}\left|\log(\varepsilon\sqrt{2})\right|\mathclose{}\Bigr]\\
    &\le{}C(\Omega,\ktwo)
\end{align*}
holds for $\varepsilon\in(0,r_{2}\slash2)$.
\end{proof}

\section{Rescaling Argument}
\label{sec:Rescaling}

Here we follow the strategy described in Proposition 2.2 of \cite{CoKePh}
and Lemma $3.1$ of \cite{BPP}.
Specifically, we rescale a solution to \eqref{def:PDE} and use the uniform
$L^{4}$ bounds from Section \ref{sec:Regularity} in order to establish 
$L^{\infty}$ and Lipschitz bounds.\\

In preparation for the rescaling argument, we define a few quantities.
We fix $x_{0}\in\overline{\Omega}$ and
$\varepsilon\in (0,\frac{1}{4}\text{inj}(\partial\Omega)]$.
This will be useful for reflection arguments needed later.
We denote the translated and scaled versions of $\Omega$ and $U$ by
\begin{equation}\label{def:rescaledU}
    \widehat{\Omega}_{x_{0},\varepsilon}\coloneqq\bigl\{z\in\mathbb{R}^{2}:x_{0}+\varepsilon{}z\in\widetilde{\Omega}\bigr\}
    =\frac{1}{\varepsilon}[\widetilde{\Omega}-x_0],
    \hspace{20pt}
    \widehat{U}_{x_{0},\varepsilon}(z)\coloneqq{}U(x_{0}+\varepsilon{}z)
\end{equation}
where $\widehat{U}_{x_{0},\varepsilon}$ is defined over
$\widehat{\Omega}_{x_{0},\varepsilon}$.
In addition, for $A\subseteq\mathbb{R}^{2}$ we introduce the notation
\begin{equation*}
    A_{x_{0},\varepsilon}\coloneqq{}
    \bigl\{z\in\mathbb{R}^{2}:x_{0}+\varepsilon{}z\in{}A\bigr\}.
\end{equation*}
We observe that
\begin{equation*}
    \partial\widehat{\Omega}_{x_{0},\varepsilon}=
    \{z\in\mathbb{R}^{2}:x_{0}+\varepsilon{}z\in\partial\widetilde{\Omega}\}.
\end{equation*}
We define
\begin{equation*}
    \widehat{B}_{R,x_{0},\varepsilon}(0)\coloneqq{}
    \widehat{\Omega}_{x_{0},\varepsilon}\cap{}B_{R}(0)
    ={}
    \widetilde{\Omega}\cap{}B_{\varepsilon{}R}(x_{0})
\end{equation*}
for $R>0$ as well as
$\widehat{\mathcal{U}}_{j,x_{0},\varepsilon}\coloneqq\frac{1}{\varepsilon}\bigl[\widetilde{\mathcal{U}}_{j}-x_{0}\bigr]$.
Finally, we define
\begin{equation*}
    \bigl<\nabla\widehat{U}_{x_{0},\varepsilon}^{i},\nabla{}v^{i}
    \bigr>_{j,\varepsilon}\coloneqq
    \biggl\{\nabla\sigma_{j}(\sigma_{j}(x_{0}+\varepsilon{}z))
    \nabla\sigma_{j}(\sigma_{j}(x_{0}+\varepsilon{}z))^{T}
    \nabla\widehat{U}_{x_{0},\varepsilon}^{i}(z)\biggr\}\cdot\nabla{}v^{i}(z)
\end{equation*}
and for appropriate $w$ we set
\begin{align*}
    &\text{div}_{j,\varepsilon}(w)(z)\\
    \coloneqq&
    \frac{1}{|\det(\nabla\sigma_{j}(x_{0}+\varepsilon{}z))|}
    \biggl[
    \widetilde{\mathcal{D}}(z)
    \partial_{\tau}\biggl(|\det(\nabla\sigma_{j}(x_{0}+\varepsilon{}z))|w(z)\cdot
    \tau(\widetilde{\psi}_{j}^{-1}(x_{0}+\varepsilon{}z))\biggr)\\
    &+\partial_{\nu}\biggl(|\det(\nabla\sigma_{j}(x_{0}+\varepsilon{}z))|
    w(z)\cdot\nu(\widetilde{\psi}_{j}^{-1}(x_{0}+\varepsilon{}z))\biggr)
    \biggr]
    \, 
\end{align*}
where
\begin{equation*}
    \widetilde{\mathcal{D}}(z)\coloneqq
    \begin{cases}
    1& \text{for }z\in\frac{1}{\varepsilon}[\widetilde{\Omega}\cap\Omega - x_0],\\
    \frac{1-(\widetilde{\psi}_{j}^{-1})^{2}(x_{0}+\varepsilon{}z)
    \kappa((\widetilde{\psi}_{j}^{-1})^{1}(x_{0}+\varepsilon{}z))}
    {1+(\widetilde{\psi}_{j}^{-1})^{2}(x_{0}+\varepsilon{}z)
    \kappa((\widetilde{\psi}_{j}^{-1})^{1}(x_{0}+\varepsilon{}z))},
    &\text{for }z\in\frac{1}{\varepsilon}[\widetilde{\Omega}\setminus\Omega - x_0],
    \end{cases}
\end{equation*}
We now begin with our first lemma demonstrating that the rescaled solution to
\eqref{eq:WeakForm} solves an elliptic PDE.

\begin{lemma}
    Suppose $u$ satisfies \eqref{eq:WeakForm},
    $U$ is the extension of $u$ defined as in
    \eqref{def:uExtension}, and $\widehat{U}$ is the rescaled form of $U$ defined in \eqref{def:rescaledU}.
    Then for $v\in{}W_{0}^{1,2}(\widehat{\mathcal{U}}_{j,x_{0},\varepsilon};\mathbb{R}^{2})$ it follows that
    \begin{align}\label{eq:RescaledPDE}
        \int_{\widehat{\mathcal{U}}_{j,x_{0},\varepsilon}}\!{}
        \biggl\{
        \kone
        \sum_{i=1}^{2}\bigl<\nabla\widehat{U}_{x_{0},\varepsilon}^{i},
        \nabla{}v^{i}\bigr>_{j,\varepsilon}
        \biggr.+&\biggl.\ktwo\,
        \emph{div}_{j,\varepsilon}(\widehat{U}_{x_{0},\varepsilon})
        \emph{div}_{j,\varepsilon}(v)
        \biggr\}
        =\int_{[\widetilde{\mathcal{U}}_{j}\cap\Omega]_{x_{0},\varepsilon}}\!{}
        \widehat{U}_{x_{0},\varepsilon}(z)\cdot{}v(z)
        (1-|\widehat{U}_{x_{0},\varepsilon}(z)|^{2})\nonumber\\
        &+\int_{[\widetilde{\mathcal{U}}_{j}\setminus\Omega]_{x_{0},\varepsilon}}\!{}
        |\det(\nabla\sigma_{j}(x_{0}+\varepsilon{}z))|
        \sigma_{j}(\widehat{U}_{x_{0},\varepsilon}(z))\cdot{}v(z)
        (1-|\widehat{U}_{x_{0},\varepsilon}(z)|^{2})\nonumber\\
        &+\varepsilon^{2}\int_{\widehat{\mathcal{U}}_{j,x_{0},\varepsilon}}\!{}\widetilde{\mathcal{F}}_{j}
        (x_{0}+\varepsilon{}z,
        \widehat{U}_{x_{0},\varepsilon}(z),
        \varepsilon^{-1}\nabla{}\widehat{U}_{x_{0},\varepsilon}(z))\cdot{}v(z)
        \, ,
    \end{align}
    In addition, for $j=1,2,\ldots,N$, $\widetilde{\mathcal{F}}_{j}(x,z,p)$ satisfies
    \begin{equation}\label{eq:normbound}
        |\widetilde{\mathcal{F}}_{j}(x,z,p)|
        \le{}C(\Omega,\ktwo)\bigl[1+|z|+|p|\bigr]
    \end{equation}
    and the operator, $\mathfrak{L}$, defined by the left-hand side of
    \eqref{eq:RescaledPDE} is elliptic.
\end{lemma}

\begin{proof}
    We notice that by \eqref{def:ExtendedPDE} from Lemma \ref{lem:PDEGluing} we have
    \begin{align*}
        &\int_{\widehat{\mathcal{U}}_{j,x_{0},\varepsilon}}\!{}
        \biggl\{
        \kone
        \sum_{i=1}^{2}\bigl<\nabla\widehat{U}_{x_{0},\varepsilon}^{i}(z),
        \nabla{}v^{i}(z)\bigr>_{j,\varepsilon}
        +\ktwo\text{div}_{j,\varepsilon}(\widehat{U}_{x_{0},\varepsilon})(z)
        \text{div}_{j,\varepsilon}(v)(z)
        \biggr\}\\
        =&\int_{\widetilde{\mathcal{U}}_{j}}\!{}
        \kone
        \sum_{i=1}^{2}\bigl<\nabla{}U^{i}(x),
        \nabla{}\bigl[v^{i}(\varepsilon^{-1}(x-x_{0}))\bigr]\bigr>_{j}
        +\int_{\widetilde{\mathcal{U}}_{j}}\!{}
        \ktwo
        \text{div}_{j}(U)(x)
        \text{div}_{j}(v([\varepsilon^{-1}(\cdot-x_{0})]))(x)\\
        =&\int_{\widetilde{\mathcal{U}}_{j}\cap\Omega}\!{}\varepsilon^{-2}
        U(x)\cdot{}v([\varepsilon^{-1}(x-x_{0})])(1-|U(x)|^{2})\\
        &+\int_{\widetilde{\mathcal{U}}_{j}\setminus\Omega}\!{}
        \varepsilon^{-2}|\det(\nabla\sigma_{j}(x))|
        \sigma_{j}(U(x))\cdot{}v([\varepsilon^{-1}(x-x_{0})])(1-|U(x)|^{2})\\
        &+\int_{\widetilde{\mathcal{U}}_{j}}\!{}\widetilde{\mathcal{F}}_{j}
        (x,U(x),\nabla{}U(x))\cdot{}v([\varepsilon^{-1}(x-x_{0})])\\
        =&\int_{\widehat{\mathcal{U}}_{j,x_{0},\varepsilon}}\!{}
        \widehat{U}_{x_{0},\varepsilon}(z)\cdot{}v(z)
        (1-|\widehat{U}_{x_{0},\varepsilon}(z)|^{2})
        +\int_{[\widetilde{\mathcal{U}}_{j}\setminus\Omega]_{x_{0},\varepsilon}}\!{}
        |\det(\nabla\sigma_{j}(x_{0}+\varepsilon{}z))|
        \sigma_{j}(\widehat{U}_{x_{0},\varepsilon}(z))\cdot{}v(z)
        (1-|\widehat{U}_{x_{0},\varepsilon}(z)|^{2})\\
        &+\varepsilon^{2}\int_{\widehat{\mathcal{U}}_{j,x_{0},\varepsilon}}\!{}\widetilde{\mathcal{F}}_{j}
        (x_{0}+\varepsilon{}z,\widehat{U}_{x_{0},\varepsilon}(z),
        \varepsilon^{-1}\nabla{}\widehat{U}_{x_{0},\varepsilon}(z))\cdot{}v(z).
    \end{align*}
    Observe that \eqref{eq:normbound} also follows from Lemma \ref{lem:PDEGluing}
    and ellipticity follows from a similar proof as Lemma
    \ref{lem:ExtensionEllipticity}.
\end{proof}

Next, we use the $O(1)-$bound from Lemma \ref{lem:Order1Potential} in order
to obtain an $O(1)-$bound for the $L^{4}$ norm.
\begin{lemma}\label{lem:L4Estimate}
    Suppose $u\in{}W_{T}^{1,2}(\Omega;\mathbb{R}^{2})$ is a minimizer of
    \eqref{def:Energy} among functions in $W_{T}^{1,2}(\Omega;\mathbb{R}^{2})$,
    $U$ is the extension of $u$ defined as in
    \eqref{def:uExtension}, and $\widehat{U}$ is the rescaled form of $U$ defined in
    \eqref{def:rescaledU}.
    Then
    \begin{equation}\label{eq:L4RecaledEstimate}
        \lVert\widehat{U}\rVert_{L^{4}(B_{1}(0);\mathbb{R}^{2})}
        \le{}C(\Omega,\ktwo).
    \end{equation}
\end{lemma}

\begin{proof}
    First observe that by Lemma \ref{lem:Order1Potential} we have
    \begin{equation*}
        \int_{\Omega}\!{}|u|^{4}
        \le2\int_{\Omega}\!{}\bigl(|u|^{2}-1\bigr)^{2}
        +2|\Omega|
        \leq 16 \, \varepsilon^{2}\,  C(\Omega,\ktwo)
        +2|\Omega|
        \le{}C(\Omega,\ktwo).
    \end{equation*}
    By appealing to a partition of unity and the coordinate representation
    of $U$ in $\widetilde{\mathcal{U}}_{j}$ for $j=1,2,\ldots,N$ we now conclude that
    \begin{equation*}
        \left\|U\right\|_{L^{4}(\widetilde{\Omega};\mathbb{R}^{2})}
        \le{}C(\Omega,\ktwo).
    \end{equation*}
    By our choice of range of $\varepsilon$ we have that
    $\widehat{B}_{1}(0)\subseteq\widetilde{\Omega}$ and hence,
    by the definition of $\widehat{U}$ and the above inequality we obtain
    \eqref{eq:L4RecaledEstimate}.
\end{proof}

Now, we show that we can find $O(1)-$bounds on the $W^{2,\frac{4}{3}}$ norm of
$\widehat{U}$.
We demonstrate this by showing that the elliptic operator $\mathfrak{L}$ can be uniformly bounded in $L^{\frac{4}{3}}$.
\begin{lemma}\label{lem:UniformBoundWidehatU}
    Suppose $u\in{}W_{T}^{1,2}(\Omega;\mathbb{R}^{2})$ is a minimizer of \eqref{def:Energy} among functions in $W_{T}^{1,2}(\Omega;\mathbb{R}^{2})$,
    $U$ is the extension of $u$ defined as in
    \eqref{def:uExtension}, and $\widehat{U}$ is the rescaled form of $U$ defined in
    \eqref{def:rescaledU}.
    Then there exists a neighbourhood $\mathcal{O}\subseteq{}\widehat{B}_{1}(0)$
    containing $0$ such that
    \begin{equation*}
        \Vert\widehat{U}\Vert_{W^{2,\frac{4}{3}}(\mathcal{O};\mathbb{R}^{2})}\le{}C(\Omega,\ktwo)
        \, .
    \end{equation*}
\end{lemma}

\begin{proof}
    By Lemma \ref{lem:L4Estimate} we have that
    \begin{equation}\label{eq:RHS_PDE_hatU_1}
        \Vert\widehat{U}(1-|\widehat{U}|^{2})\Vert
        _{L^{\frac{4}{3}}(\widehat{B}_{1}(0);\mathbb{R}^{2})}
        \le{}C(\Omega,\ktwo)
    \end{equation}
    as well as
    \begin{equation}\label{eq:RHS_PDE_hatU_2}
        \left\||\det(\nabla\sigma_{j}(x_{0}+\varepsilon{}(\cdot)))|
        \widehat{U}(1-|\widehat{U}|^{2})\right\|
        _{L^{\frac{4}{3}}(\widehat{B}_{1}(0);\mathbb{R}^{2})}
        \le{}C(\Omega,\ktwo).
    \end{equation}
    In addition, we have that
    \begin{equation}\label{eq:RHS_PDE_hatU_3}
        \varepsilon^{2}
        \Vert\widetilde{\mathcal{F}}(x_{0}+\varepsilon{}z,\widehat{U}(z),
        \varepsilon^{-1}\nabla\widehat{U}(z))\Vert_{L^{\frac{4}{3}}
        (\widehat{B}_{1}(0);\mathbb{R}^{2})}
        \le{}C(\Omega,\ktwo)\varepsilon
        \mathopen{}|\log(\varepsilon)|\mathclose{}.
    \end{equation}
    Since the RHS of the PDE satisfied by $\widehat{U}$ in \eqref{eq:RescaledPDE} can be estimated using \eqref{eq:RHS_PDE_hatU_1},
    \eqref{eq:RHS_PDE_hatU_2} and \eqref{eq:RHS_PDE_hatU_3},
    we get that by regularity for elliptic systems in $L^p$ and duality arguments (see \cite[Section 4.3]{Gi93}
    combined with standard localization arguments)
    that the LHS of that PDE satisfies    
    \begin{equation*}
        \Vert\mathfrak{L}(\widehat{U})\Vert_{L^{\frac{4}{3}}(\mathcal{O};
        \mathbb{R}^{2})}\le{}C(\Omega,\ktwo).
    \end{equation*}
    By interior regularity we obtain the desired regularity of $\widehat{U}$
    on $\mathcal{O}\subseteq\widehat{B}_{1}(0)$.
\end{proof}

Next we provide a bootstrapping argument to show that we can obtain uniform Lipschitz bounds for $\widehat{U}$.
\begin{lemma}\label{lem:LipschtzBound}
    Suppose $u\in{}W_{T}^{1,2}(\Omega;\mathbb{R}^{2})$ is a minimizer of
    \eqref{def:Energy} among functions in $W_{T}^{1,2}(\Omega;\mathbb{R}^{2})$,
    $U$ is the extension of $u$ defined as in
    \eqref{def:uExtension}, and $\widehat{U}$ is the rescaled form of $U$ defined in
    \eqref{def:rescaledU}.
    Then there exists a neighbourhood $\mathcal{O}\subseteq{}\widehat{B}_{1}(0)$
    containing $0$ such that
    \begin{equation*}
        \Vert\widehat{U}\Vert_{L^{\infty}(\mathcal{O};\mathbb{R}^{2})}\le{}C(\Omega,\ktwo)
        \, .
    \end{equation*}
    If $\Omega$ has $C^{3,1}-$boundary then we also have
    \begin{equation*}
        \Vert\widehat{U}\Vert_{C^{0,1}(\mathcal{O};\mathbb{R}^{2})}\le{}C(\Omega,\ktwo)
        \, .
    \end{equation*}
\end{lemma}

\begin{proof}
    By Lemma \ref{lem:UniformBoundWidehatU} we have that
    $\widehat{U}\in{}W^{2,\frac{4}{3}}(\widehat{B}_{1}(0);\mathbb{R}^{2})$
    and hence by the Sobolev-Embedding Theorem and Morrey's inequality we have
    \begin{equation*}
        \lVert\widehat{U}\rVert_{L^{\infty}(\mathcal{O};\mathbb{R}^{2})}
        \le{}C\lVert\widehat{U}\rVert_{W^{1,4}(\mathcal{O};\mathbb{R}^{2})}
        \le{}C\lVert\widehat{U}\rVert_{W^{2,\frac{4}{3}}(\mathcal{O};\mathbb{R}^{2})}
        \le{}C(\Omega,\ktwo).
    \end{equation*}
    Next we show that we can control the $L^{\infty}-$norm of the gradient
    provided $\Omega$ has $C^{3,1}-$boundary.
    Observe that since $\widehat{U}\in{}W^{1,4}(\mathcal{O};\mathbb{R}^{2})$
    then
    \begin{equation*}
        \Vert\widehat{U}(1-|\widehat{U}|^{2})\Vert
        _{W^{1,\frac43}(\mathcal{O};\mathbb{R}^{2})}
        \le{}C(\Omega,\ktwo)
    \end{equation*}
    and
    \begin{equation*}
        \left\||\det(\nabla\sigma_{j}(x_{0}+\varepsilon{}(\cdot)))|
        \widehat{U}(1-|\widehat{U}|^{2})\right\|
        _{W^{1,\frac43}(\mathcal{O};\mathbb{R}^{2})}
        \le{}C(\Omega,\ktwo).
    \end{equation*}
    In addition, we have
    \begin{equation*}
        \varepsilon^{2}\left\|\widetilde{\mathcal{F}}(x_{0}+\varepsilon{}z,
        \widehat{U}(z),\varepsilon^{-1}\nabla\widehat{U}(z))\right\|_{
        W^{1,\frac{4}{3}}(\mathcal{O};\mathbb{R}^{2})}\le{}C(\Omega,\ktwo).
    \end{equation*}
    Differentiating both sides of our elliptic PDE we now obtain (appealing
    again to Section 4.3 of \cite{Gi93}) 
    that $\widehat{U}\in{}W^{3,\frac{4}{3}}(\mathcal{O};\mathbb{R}^{2})$.
    By the Sobolev-Embedding Theorem and Morrey's inequality we now have
    \begin{equation*}
        \lVert\widehat{U}\rVert_{C^{0,1}(\mathcal{O};\mathbb{R}^{2})}
        \le{}C\lVert\widehat{U}\rVert_{W^{2,4}(\mathcal{O};\mathbb{R}^{2})}
        \le{}C\lVert\widehat{U}\rVert_{W^{3,\frac{4}{3}}(\mathcal{O};\mathbb{R}^{2})}
        \le{}C(\Omega,\ktwo).
    \end{equation*}
\end{proof}

\noindent{}Finally, we provide a proof of Theorem \ref{thm:main}.\\[5pt]
{\it Proof of Theorem \ref{thm:main}}\\[3pt]
Since $x_{0}\in\overline{\Omega}$ was arbitrary then the $L^{\infty}$ bound
follows from Lemma \ref{lem:LipschtzBound} and compactness of
$\overline{\Omega}$.
Similarly, when $\Omega$ has $C^{3,1}-$boundary, since
$\nabla\widehat{U}(z)=\varepsilon\nabla{}u(x_{0}+\varepsilon{}z)$ then
the Lipschitz bound follows from Lemma \ref{lem:LipschtzBound} and
compactness of $\overline{\Omega}$.
\qed

\appendix
\appendixpage
\addappheadtotoc

\section{Calculations}
\label{app:calculations}
In this appendix, we provide comprehensive calculations regarding the PDE
gluing from Lemma \ref{lem:PDEGluing}.\\

\noindent{}{\it Lemma 3.1 calculations}\\[3pt]
Since $u$ satisfies \eqref{eq:WeakForm} then for $v\in{}W_{T}^{1,2}(\Omega;\mathbb{R}^{2})$
    satisfying $\text{supp}(v)\subseteq{}\overline{\mathcal{U}}_{j}$ we have, after writing
    $u(x)=
    \widetilde{u}_{\tau}\bigl(\widetilde{\psi}_{j}^{-1}(x)\bigr)
    \tau\bigl((\widetilde{\psi}_{j}^{-1})^{1}(x)\bigr)
    +\widetilde{u}_{\nu}\bigl(\widetilde{\psi}_{j}^{-1}(x)\bigr)
    \nu\bigl((\widetilde{\psi}_{j}^{-1})^{1}(x)\bigr)$, and similarly for
    $v$, as well as integrating by parts that
    \begin{align*}
        &\kone\int_{\mathcal{U}_{j}}\!{}
        \Bigl[\nabla\widetilde{\psi}_{j}^{-1}(x)^{T}
        (\nabla{}\widetilde{u}_{\tau})\bigl(\widetilde{\psi}_{j}^{-1}(x)\bigr)\Bigr]\cdot
        \Bigl[\nabla\widetilde{\psi}_{j}^{-1}(x)^{T}
        \nabla{}\widetilde{v}_{\tau}\bigl(\widetilde{\psi}_{j}^{-1}(x)\bigr)\Bigr]\\
        &+\kone\int_{\mathcal{U}_{j}}\!{}
        \Bigl[\nabla\widetilde{\psi}_{j}^{-1}(x)^{T}
        (\nabla{}\widetilde{u}_{\nu})\bigl(\widetilde{\psi}_{j}^{-1}(x)\bigr)\Bigr]\cdot
        \Bigl[\nabla\widetilde{\psi}_{j}^{-1}(x)^{T}
        \nabla{}\widetilde{v}_{\nu}\bigl(\widetilde{\psi}_{j}^{-1}(x)\bigr)\Bigr]\\
        &-\kone\int_{\mathcal{U}_{j}}\!{}
        \Bigl[\nabla\widetilde{\psi}_{j}^{-1}(x)^{T}
        (\nabla{}\widetilde{u}_{\tau})\bigl(\widetilde{\psi}_{j}^{-1}(x)\bigr)\Bigr]\cdot
        \biggl[\frac{\kappa\bigl((\widetilde{\psi}_{j}^{-1})^{1}(x)\bigr)
        \widetilde{v}_{\nu}(\widetilde{\psi}_{j}^{-1}(x))}
        {1-(\widetilde{\psi}_{j}^{-1})^{2}(x)\kappa\bigl((\widetilde{\psi}_{j}^{-1})^{1}(x)\bigr)}
        \tau\bigl(\widetilde{\psi}_{j}^{-1}(x)\bigr)\biggr]\\
        &+\kone\int_{\mathcal{U}_{j}}\!{}
        \Bigl[\nabla\widetilde{\psi}_{j}^{-1}(x)^{T}
        (\nabla{}\widetilde{u}_{\nu})\bigl(\widetilde{\psi}_{j}^{-1}(x)\bigr)\Bigr]\cdot
        \biggl[\frac{\kappa\bigl((\widetilde{\psi}_{j}^{-1})^{1}(x)\bigr)
        \widetilde{v}_{\tau}\bigl(\widetilde{\psi}_{j}^{-1}(x)\bigr)}
        {1-(\widetilde{\psi}_{j}^{-1})^{2}(x)\kappa\bigl((\widetilde{\psi}_{j}^{-1})^{1}(x)\bigr)}
        \nu\bigl(\widetilde{\psi}_{j}^{-1}(x)\bigr)\biggr]\\
        &+\sum_{i=1}^{2}\kone\int_{\mathcal{U}_{j}}\!{}\text{div}\biggl[
        \frac{\widetilde{u}_{\nu}\bigl(\widetilde{\psi}_{j}^{-1}(x)\bigr)
        \kappa\bigl((\widetilde{\psi}_{j}^{-1})^{1}(x)\bigr)}
        {1-(\widetilde{\psi}_{j}^{-1})^{2}(x)\kappa\bigl((\widetilde{\psi}_{j}^{-1})^{1}(x)\bigr)}
        \tau^{i}\bigl((\widetilde{\psi}_{j}^{-1})^{1}(x)\bigr)
        \tau\bigl((\widetilde{\psi}_{j}^{-1})^{1}(x)\bigr)\biggr]
        v^{i}(x)\\
        &-\sum_{i=1}^{2}\kone\int_{\mathcal{U}_{j}}\!{}\text{div}\biggl[
        \frac{\widetilde{u}_{\tau}\bigl(\widetilde{\psi}_{j}^{-1}(x)\bigr)
        \kappa\bigl((\widetilde{\psi}_{j}^{-1})^{1}(x)\bigr)}
        {1-(\widetilde{\psi}_{j}^{-1})^{2}(x)\kappa\bigl((\widetilde{\psi}_{j}^{-1})^{1}(x)\bigr)}
        \nu^{i}\bigl((\widetilde{\psi}_{j}^{-1})^{1}(x)\bigr)
        \tau\bigl((\widetilde{\psi}_{j}^{-1})^{1}(x)\bigr)\biggr]
        v^{i}(x)\\
        &+\ktwo\int_{\mathcal{U}_{j}}
        \Bigl[
        \partial_{y_{1}}\widetilde{u}_{\tau}\bigl(\widetilde{\psi}_{j}^{-1}(x)\bigr)
        \partial_{\tau}(\widetilde{\psi}_{j}^{-1})^{1}(x)
        +\partial_{y_{2}}\widetilde{u}_{\nu}\bigl(\widetilde{\psi}_{j}^{-1}(x)\bigr)
        \partial_{\nu}(\widetilde{\psi}_{j}^{-1})^{2}(x)
        \Bigr]
        \Bigl[\partial_{y_{1}}\widetilde{v}_{\tau}\bigl(\widetilde{\psi}_{j}^{-1}(x)\bigr)
        \partial_{\tau}(\widetilde{\psi}_{j}^{-1})^{1}(x)\Bigr]\\
        &+\ktwo\int_{\mathcal{U}_{j}}
        \Bigl[
        \partial_{y_{1}}\widetilde{u}_{\tau}\bigl(\widetilde{\psi}_{j}^{-1}(x)\bigr)
        \partial_{\tau}(\widetilde{\psi}_{j}^{-1})^{1}(x)
        +\partial_{y_{2}}\widetilde{u}_{\nu}\bigl(\widetilde{\psi}_{j}^{-1}(x)\bigr)
        \partial_{\nu}(\widetilde{\psi}_{j}^{-1})^{2}(x)
        \Bigr]
        \Bigl[\partial_{y_{2}}\widetilde{v}_{\nu}\bigl(\widetilde{\psi}_{j}^{-1}(x)\bigr)
        \partial_{\nu}(\widetilde{\psi}_{j}^{-1})^{2}(x)\Bigr]\\
        &+\ktwo\int_{\mathcal{U}_{j}}\!{}
        \Bigl[
        \partial_{y_{1}}\widetilde{u}_{\tau}\bigl(\widetilde{\psi}_{j}^{-1}(x)\bigr)
        \partial_{\tau}(\widetilde{\psi}_{j}^{-1})^{1}(x)
        +\partial_{y_{2}}\widetilde{u}_{\nu}\bigl(\widetilde{\psi}_{j}^{-1}(x)\bigr)
        \partial_{\nu}(\widetilde{\psi}_{j}^{-1})^{2}(x)
        \Bigr]\Bigl[\widetilde{v}_{\nu}(\widetilde{\psi}_{j}^{-1}(x))
        \text{div}\bigl(\nu\bigl((\widetilde{\psi}_{j}^{-1})^{1}(x)\bigr)\bigr)\Bigr]
        \\
        &-\ktwo\int_{\mathcal{U}_{j}}\!{}
        \nabla\Big[
        \widetilde{u}_{\nu}\bigl(\widetilde{\psi}_{j}^{-1}(x)\bigr)
        \text{div}\bigl(\nu\bigl((\widetilde{\psi}_{j}^{-1})^{1}(x)\bigr)\bigr)\Bigr]\cdot
        \Bigl[\widetilde{v}_{\tau}\bigl(\widetilde{\psi}_{j}^{-1}(x)\bigr)
        \tau\bigl((\widetilde{\psi}_{j}^{-1})^{1}(x)\bigr)+
        \widetilde{v}_{\nu}\bigl(\widetilde{\psi}_{j}^{-1}(x)\bigr)
        \nu\bigl((\widetilde{\psi}_{j}^{-1})^{1}(x)\bigr)\Bigr]
        \\
        =&\int_{\mathcal{U}_{j}}\!{}
        \frac{u\cdot{}v}{\varepsilon^{2}}(1-|u|^{2}).
    \end{align*}
    Notice that we have used that $\widetilde{u}_{\nu}=0=\widetilde{v}_{\nu}$ on $\partial\Omega$
    when integrating by parts.
    Notice also that since the tangential part of $v$ and the normal part of $v$
    vary independently then, after grouping many of the terms together as well as
    using that
    \begin{equation*}
        \mathcal{G}_{j}(x,\widetilde{\psi}_{j}^{-1}(x))=
        \nabla\widetilde{\psi}_{j}^{-1}(x)\nabla\widetilde{\psi}_{j}^{-1}(x)^{T}=
        \begin{pmatrix}
            \frac{1}
            {\bigl(1-(\widetilde{\psi}_{j}^{-1})^{2}(x)\kappa\bigl((\widetilde{\psi}_{j}^{-1})^{1}(x)\bigr)\bigr)^{2}}& 0\\
            0& 1
        \end{pmatrix}
    \end{equation*}
    we have
    \begin{align}
        &\kone\int_{\mathcal{U}_{j}}\!{}
        \biggl[
        \frac{\partial_{y_{1}}\widetilde{u}_{\tau}\bigl(\widetilde{\psi}_{j}^{-1}(x)\bigr)
        \partial_{y_{1}}\widetilde{v}_{\tau}\bigl(\widetilde{\psi}_{j}^{-1}(x)\bigr)}
        {\bigl(1-(\widetilde{\psi}_{j}^{-1})^{2}(x)\kappa\bigl((\widetilde{\psi}_{j}^{-1})^{1}(x)\bigr)\bigr)^{2}}
        +
        \partial_{y_{2}}\widetilde{u}_{\tau}\bigl(\widetilde{\psi}_{j}^{-1}(x)\bigr)
        \partial_{y_{2}}\widetilde{v}_{\tau}\bigl(\widetilde{\psi}_{j}^{-1}(x)\bigr)
        \biggr]
        \nonumber\\
        &+\ktwo\int_{\mathcal{U}_{j}}
        \biggl[
        \frac{
        \partial_{y_{1}}\widetilde{u}_{\tau}\bigl(\widetilde{\psi}_{j}^{-1}(x)\bigr)
        }{1-(\widetilde{\psi}_{j}^{-1})^{2}(x)
        \kappa\bigl((\widetilde{\psi}_{j}^{-1})^{1}(x)\bigr)}
        +\partial_{y_{2}}\widetilde{u}_{\nu}\bigl(\widetilde{\psi}_{j}^{-1}(x)\bigr)
        \biggr]
        \biggl[\frac{
        \partial_{y_{1}}\widetilde{v}_{\tau}\bigl(\widetilde{\psi}_{j}^{-1}(x)\bigr)
        }{1-(\widetilde{\psi}_{j}^{-1})^{2}(x)
        \kappa\bigl((\widetilde{\psi}_{j}^{-1})^{1}(x)\bigr)}
        \biggr]
        \nonumber\\
        =&\int_{\mathcal{U}_{j}}\!{}
        \frac{\widetilde{u}_{\tau}\bigl(\widetilde{\psi}_{j}^{-1}(x)\bigr)
        \widetilde{v}_{\tau}\bigl(\widetilde{\psi}_{j}^{-1}(x)\bigr)}
        {\varepsilon^{2}}(1-|u|^{2})
        +\int_{\mathcal{U}_{j}}\!{}F_{j,\tau}\bigl(x,u(x),\nabla{}u(x)\bigr)
        \widetilde{v}_{\tau}\bigl(\widetilde{\psi}_{j}^{-1}(x)\bigr)
        \label{eq:TangentialPDEApp}
    \end{align}
    and
    \begin{align}
        &\kone\int_{\mathcal{U}_{j}}\!{}
       \biggl[
        \frac{\partial_{y_{1}}\widetilde{u}_{\nu}\bigl(\psi_{j}^{-1}(x)\bigr)
        \partial_{y_{1}}\widetilde{v}_{\nu}\bigl(\psi_{j}^{-1}(x)\bigr)}
        {\bigl(1-(\psi_{j}^{-1})^{2}(x)\kappa\bigl((\psi_{j}^{-1})^{1}(x)\bigr)\bigr)^{2}}
        +
        \partial_{y_{2}}\widetilde{u}_{\nu}\bigl(\psi_{j}^{-1}(x)\bigr)
        \partial_{y_{2}}\widetilde{v}_{\nu}\bigl(\psi_{j}^{-1}(x)\bigr)
        \biggr]\nonumber\\
        &+\ktwo\int_{\mathcal{U}_{j}}
        \biggl[\frac{
        \partial_{y_{1}}\widetilde{u}_{\tau}\bigl(\psi_{j}^{-1}(x)\bigr)
        }{1-(\widetilde{\psi}_{j}^{-1})^{2}(x)
        \kappa\bigl((\widetilde{\psi}_{j}^{-1})^{1}(x)\bigr)}
        +\partial_{y_{2}}\widetilde{u}_{\nu}\bigl(\psi_{j}^{-1}(x)\bigr)
        \biggr]
        \Bigl[\partial_{y_{2}}\widetilde{v}_{\nu}\bigl(\psi_{j}^{-1}(x)\bigr)
        \Bigr]\nonumber\\
        =&\int_{\mathcal{U}_{j}}\!{}
        \frac{\widetilde{u}_{\nu}\bigl(\psi_{j}^{-1}(x)\bigr)
        \widetilde{v}_{\nu}\bigl(\psi_{j}^{-1}(x)\bigr)}
        {\varepsilon^{2}}(1-|u|^{2})
        +\int_{\mathcal{U}_{j}}\!{}F_{j,\nu}\bigl(x,u(x),\nabla{}u(x)\bigr)
        \widetilde{v}_{\nu}\bigl(\psi_{j}^{-1}(x)\bigr)
        \, ,
        \label{eq:NormalPDEApp}
    \end{align}
    where $F_{j,\tau}$ and $F_{j,\nu}$ are determined by the remaining integrands from the previous
    calculation.
    Notice that $F_{j,\tau}$ and $F_{j,\nu}$ satisfy
    \begin{equation}\label{eq:carathmapApp}
        \max\{|F_{j,\tau}(x,z,p)|,|F_{j,\nu}(x,z,p)|\}
        \le{}C(\Omega,\ktwo)\bigl[1+|z|+|p|\bigr]
    \end{equation}
    for each $j=1,2\ldots,N$.\\
    
    Next, we use this computation in order to show that the extension, $U$,
    satisfies \eqref{def:ExtendedPDE} on $\widetilde{\mathcal{U}}_{j}$ for $j=0,1,2\ldots,N$.
    Since $U=u$ on $\widetilde{\mathcal{U}}_{0}$ then we may restrict attention to
    the case of $j=1,2,\ldots,N$.
    for $x\in\widetilde{\mathcal{U}}_{j}$ for $j=1,2,\ldots,N$.
    To obtain our goal we first observe that, by \eqref{def:uExtension}, we have,
    for $x\in\mathcal{\widetilde{U}}_{j}\setminus\Omega$, that
    \begin{align*}
        &\nabla\sigma_{j}\bigl(\sigma_{j}(x)\bigr)
        \nabla\sigma_{j}\bigl(\sigma_{j}(x)\bigr)
        \nabla{}U^{i}(x)\\
        =&\nabla\widetilde{\psi}_{j}\bigl(\widetilde{\psi}_{j}^{-1}(x)\bigr)
        \bigl[R\mathcal{G}_{j}(x,R\widetilde{\psi}_{j}^{-1}(x))\bigr]
        \nabla\widetilde{u}_{\tau}\bigl(R\widetilde{\psi}_{j}^{-1}(x)\bigr)
        \tau^{i}\bigl((\widetilde{\psi}_{j}^{-1})^{1}(x)\bigr)\\
        &-\nabla\widetilde{\psi}_{j}\bigl(\widetilde{\psi}_{j}^{-1}(x)\bigr)
        \bigl[R\mathcal{G}_{j}(x,R\widetilde{\psi}_{j}^{-1}(x))\bigr]
        \nabla\widetilde{u}_{\nu}\bigl(R\widetilde{\psi}_{j}^{-1}(x)\bigr)
        \nu^{i}\bigl((\widetilde{\psi}_{j}^{-1})^{1}(x)\bigr)\\
        &+\frac{\kappa\bigl((\widetilde{\psi}_{j}^{-1})^{1}(x)\bigr)}
        {\bigl(1+(\widetilde{\psi}_{j}^{-1})^{2}(x)\kappa\bigl((\widetilde{\psi}_{j}^{-1})^{1}(x)\bigr)\bigr)^{2}}
        \widetilde{u}_{\tau}\bigl(R(\widetilde{\psi}_{j}^{-1})^{1}(x)\bigr)
        \nabla\widetilde{\psi}_{j}\bigl(\widetilde{\psi}_{j}^{-1}(x)\bigr)
        \nu^{i}\bigl((\widetilde{\psi}_{j}^{-1})^{1}(x)\bigr)
        \mathbf{e}_{1}\\
        &+\frac{\kappa\bigl((\widetilde{\psi}_{j}^{-1})^{1}(x)\bigr)}
        {\bigl(1+(\widetilde{\psi}_{j}^{-1})^{2}(x)\kappa\bigl((\widetilde{\psi}_{j}^{-1})^{1}(x)\bigr)\bigr)^{2}}
        \widetilde{u}_{\nu}\bigl(R(\widetilde{\psi}_{j}^{-1})^{1}(x)\bigr)
        \nabla\widetilde{\psi}_{j}\bigl(\widetilde{\psi}_{j}^{-1}(x)\bigr)
        \tau^{i}\bigl((\widetilde{\psi}_{j}^{-1})^{1}(x)\bigr)
        \mathbf{e}_{1}.
    \end{align*}
    From this we conclude, using a local tangent-normal decomposition of $v$ for
    $x\in\widetilde{\mathcal{U}}_{j}\setminus\Omega$, that
    \begin{align*}
        &\sum_{i=1}^{2}
        \Bigl[\nabla\sigma_{j}\bigl(\sigma_{j}(x)\bigr)
        \nabla\sigma_{j}\bigl(\sigma_{j}(x)\bigr)^{T}
        \nabla{}U^{i}(x)\Bigr]\cdot{}\nabla{}v^{i}(x)\\
        =&\nabla\widetilde{u}_{\tau}\bigl(R\widetilde{\psi}_{j}^{-1}(x)\bigr)^{T}
        \bigl[R\mathcal{G}_{j}(x,R\widetilde{\psi}_{j}^{-1}(x))\bigr]
        \nabla\widetilde{v}_{\tau}\bigl(\widetilde{\psi}_{j}^{-1}(x)\bigr)\\
        &-\nabla\widetilde{u}_{\nu}\bigl(R\widetilde{\psi}_{j}^{-1}(x)\bigr)^{T}
        \bigl[R\mathcal{G}_{j}(x,R\widetilde{\psi}_{j}^{-1}(x))\bigr]
        \nabla\widetilde{v}_{\nu}\bigl(\widetilde{\psi}_{j}^{-1}(x)\bigr)\\
        &+\kappa\bigl((\widetilde{\psi}_{j}^{-1})^{1}(x)\bigr)
        \widetilde{v}_{\nu}\bigl(\widetilde{\psi}_{j}^{-1}(x)\bigr)
        \nabla\widetilde{u}_{\tau}\bigl(R\widetilde{\psi}_{j}^{-1}(x)\bigr)^{T}
        \bigl[R\mathcal{G}_{j}(x,R\widetilde{\psi}_{j}^{-1}(x))\bigr]
        \mathbf{e}_{1}\\
        &-\kappa\bigl((\widetilde{\psi}_{j}^{-1})^{1}(x)\bigr)
        \widetilde{v}_{\tau}\bigl(\widetilde{\psi}_{j}^{-1}(x)\bigr)
        \nabla\widetilde{u}_{\nu}\bigl(R\widetilde{\psi}_{j}^{-1}(x)\bigr)^{T}
        \bigl[R\mathcal{G}_{j}(x,R\widetilde{\psi}_{j}^{-1}(x))\bigr]
        \mathbf{e}_{1}\\
        &+\sum_{i=1}^{2}
        \frac{\kappa\bigl((\widetilde{\psi}_{j}^{-1})^{1}(x)\bigr)
        \nu^{i}\bigl((\widetilde{\psi}_{j}^{-1})^{1}(x)\bigr)}
        {\bigl(1+(\widetilde{\psi}_{j}^{-1})^{2}(x)\kappa\bigl((\widetilde{\psi}_{j}^{-1})^{1}(x)\bigr)\bigr)^{2}}
        \widetilde{u}_{\tau}\bigl((\widetilde{\psi}_{j}^{-1})^{1}(x)\bigr)
        \Bigl[
        \nabla\widetilde{\psi}_{j}\bigl(\widetilde{\psi}_{j}^{-1}(x)\bigr)\mathbf{e}_{1}\Bigr]\cdot{}\nabla{}v(x)\\
        &+\sum_{i=1}^{2}
        \frac{\kappa\bigl((\widetilde{\psi}_{j}^{-1})^{1}(x)\bigr)
        \tau^{i}\bigl((\widetilde{\psi}_{j}^{-1})^{1}(x)\bigr)}
        {\bigl(1+(\widetilde{\psi}_{j}^{-1})^{2}(x)\kappa\bigl((\widetilde{\psi}_{j}^{-1})^{1}(x)\bigr)\bigr)^{2}}
        \widetilde{u}_{\nu}\bigl((\widetilde{\psi}_{j}^{-1})^{1}(x)\bigr)
        \Bigl[
        \nabla\widetilde{\psi}_{j}\bigl(\widetilde{\psi}_{j}^{-1}(x)\bigr)\mathbf{e}_{1}\Bigr]\cdot{}\nabla{}v(x).
    \end{align*}
    A similar expression holds over $\widetilde{\mathcal{U}}_{j}\cap\Omega$
    due to the earlier calculation for $u$.
    Notice that the last two terms, containing a gradient of $v$ but not of $u$, will have no 
    boundary terms after integrating by parts since $\widetilde{u}_{\nu}=0$
    on $\partial\Omega$, $v=0$ on $\partial\widetilde{\mathcal{U}}_{j}$, and since
    the integrals involving $\widetilde{u}_{\tau}$ and $v$ will cancel after
    an integration by parts.
    Multiplying by $|\det(\nabla\sigma_{j}(x))|$, integrating over
    $\widetilde{\mathcal{U}}_{j}\setminus\Omega$, applying the Change of
    Variables theorem with $\sigma_{j}$, and using that $\sigma_{j}^{2}=\text{Id}$
    we obtain
    \begin{align*}
        &
        \int_{\widetilde{\mathcal{U}}_{j}\setminus\Omega}\!{}|\det(\nabla\sigma_{j}(x))|
        \nabla\widetilde{u}_{\tau}\bigl(R\widetilde{\psi}_{j}^{-1}(x)\bigr)^{T}
        \bigl[R\mathcal{G}_{j}(x,R\widetilde{\psi}_{j}^{-1}(x))\bigr]
        \nabla\widetilde{v}_{\tau}\bigl(\widetilde{\psi}_{j}^{-1}(x)\bigr)\\
        &-
        \int_{\widetilde{\mathcal{U}}_{j}\setminus\Omega}\!{}|\det(\nabla\sigma_{j}(x))|
        \nabla\widetilde{u}_{\nu}\bigl(R\widetilde{\psi}_{j}^{-1}(x)\bigr)^{T}
        \bigl[R\mathcal{G}_{j}(x,R\widetilde{\psi}_{j}^{-1}(x))\bigr]
        \nabla\widetilde{v}_{\nu}\bigl(\widetilde{\psi}_{j}^{-1}(x)\bigr)\\
        =&\int_{\widetilde{\mathcal{U}}_{j}\cap\Omega}\!{}
        \nabla\widetilde{u}_{\tau}\bigl(\widetilde{\psi}_{j}^{-1}(x)\bigr)
        \bigl[R\mathcal{G}_{j}(x,\widetilde{\psi}_{j}^{-1}(x))\bigr]
        \nabla\widetilde{v}_{\tau}\bigl(R\widetilde{\psi}_{j}^{-1}(x)\bigr)\\
        &-\int_{\widetilde{\mathcal{U}}_{j}\cap\Omega}\!{}
        \nabla\widetilde{u}_{\nu}\bigl(\widetilde{\psi}_{j}^{-1}(x)\bigr)
        \bigl[R\mathcal{G}_{j}(x,\widetilde{\psi}_{j}^{-1}(x))\bigr]
        \nabla\widetilde{v}_{\nu}\bigl(R\widetilde{\psi}_{j}^{-1}(x)\bigr)\\
        =&\int_{\widetilde{\mathcal{U}}_{j}\cap\Omega}\!{}
        \nabla\widetilde{u}_{\tau}\bigl(\widetilde{\psi}_{j}^{-1}(x)\bigr)
        \bigl[\mathcal{G}_{j}(x,\widetilde{\psi}_{j}^{-1}(x))\bigr]
        \nabla(\widetilde{v}^{R})_{\tau}\bigl(\widetilde{\psi}_{j}^{-1}(x)\bigr)\\
        &-\int_{\widetilde{\mathcal{U}}_{j}\cap\Omega}\!{}
        \nabla\widetilde{u}_{\nu}\bigl(\widetilde{\psi}_{j}^{-1}(x)\bigr)
        \bigl[\mathcal{G}_{j}(x,\widetilde{\psi}_{j}^{-1}(x))\bigr]
        \nabla(\widetilde{v}^{R})_{\nu}\bigl(\widetilde{\psi}_{j}^{-1}(x)\bigr)
    \end{align*}
    where $\widetilde{v}^{R}(y)\coloneqq{}\widetilde{v}(Ry)$.
    Next we see by the definitions of $U$ from \eqref{def:uExtension}
    for $x\in\widetilde{\mathcal{U}}_{j}\setminus\Omega$ and of
    $\text{div}_{j}$ that
    \begin{align*}
        \text{div}_{j}(U)(x)=&
        |\det(\nabla\sigma_{j}(x))|^{\frac{1}{2}}\biggl[
        \frac{1-(\widetilde{\psi}_{j}^{-1})^{2}(x)\kappa\bigl((\widetilde{\psi}_{j}^{-1})^{1}(x)\bigr)}
        {1+(\widetilde{\psi}_{j}^{-1})^{2}(x)\kappa\bigl((\widetilde{\psi}_{j}^{-1})^{1}(x)\bigr)}
        \partial_{\tau}\Bigl[
        \widetilde{u}_{\tau}\bigl(R\widetilde{\psi}_{j}^{-1}(x)\bigr)\Bigr]
        -\partial_{\nu}\Bigl[
        \widetilde{u}_{\nu}\bigl(R\widetilde{\psi}_{j}^{-1}(x)\bigr)\Bigr]
        \biggr]\\
        =&|\det(\nabla\sigma_{j}(x))|^{\frac{1}{2}}\biggl[
        \frac{\partial_{y_{1}}\widetilde{u}_{\tau}\bigl(R\widetilde{\psi}_{j}^{-1}(x)\bigr)}
        {1+(\widetilde{\psi}_{j}^{-1})^{2}(x)\kappa\bigl((\widetilde{\psi}_{j}^{-1})^{1}(x)\bigr)}
        +\partial_{y_{2}}\widetilde{u}_{\nu}\bigl(R\widetilde{\psi}_{j}^{-1}(x)\bigr)\biggr].
    \end{align*}
    A similar computation holds for $v$.
    From this we see, after integrating over
    $\widetilde{\mathcal{U}}_{j}\setminus\Omega$, that
    \begin{equation*}
        \int_{\widetilde{\mathcal{U}}_{j}\setminus\Omega}\!{}
        \text{div}_{j}(U)\text{div}_{j}(v)
        =\int_{\widetilde{\mathcal{U}}_{j}\setminus\Omega}\!{}
        |\det(\nabla\sigma_{j}(x))|^{\frac{1}{2}}\biggl[
        \frac{\partial_{y_{1}}\widetilde{u}_{\tau}\bigl(R\widetilde{\psi}_{j}^{-1}(x)\bigr)}
        {1+(\widetilde{\psi}_{j}^{-1})^{2}(x)\kappa\bigl((\widetilde{\psi}_{j}^{-1})^{1}(x)\bigr)}
        +\partial_{y_{2}}\widetilde{u}_{\nu}\bigl(R\widetilde{\psi}_{j}^{-1}(x)\bigr)
        \biggr]
        \text{div}_{j}(v).
    \end{equation*}
    Changing variables now gives
    \begin{align*}
        &\int_{\widetilde{\mathcal{U}}_{j}\cap\Omega}\!{}
        |\det(\nabla\sigma_{j}(x))|^{\frac{-1}{2}}
        \biggl[
        \frac{\partial_{y_{1}}\widetilde{u}_{\tau}\bigl(\widetilde{\psi}_{j}^{-1}(x)\bigr)}
        {1-(\widetilde{\psi}_{j}^{-1})^{2}(x)\kappa\bigl((\widetilde{\psi}_{j}^{-1})^{1}(x)\bigr)}
        +\partial_{y_{2}}\widetilde{u}_{\nu}\bigl(\widetilde{\psi}_{j}^{-1}(x)\bigr)
        \biggr]
        \text{div}_{j}(v)(\sigma_{j}(x))\\
        =&\int_{\widetilde{\mathcal{U}}_{j}\cap\Omega}\!{}
        \biggl[
        \frac{\partial_{y_{1}}\widetilde{u}_{\tau}\bigl(\widetilde{\psi}_{j}^{-1}(x)\bigr)}
        {1-(\widetilde{\psi}_{j}^{-1})^{2}(x)\kappa\bigl((\widetilde{\psi}_{j}^{-1})^{1}(x)\bigr)}
        +\partial_{y_{2}}\widetilde{u}_{\nu}\bigl(\widetilde{\psi}_{j}^{-1}(x)\bigr)
        \biggr]
        \frac{\partial_{y_{1}}\widetilde{v}_{\tau}\bigl(R\widetilde{\psi}_{j}^{-1}(x)\bigr)}
        {1-(\widetilde{\psi}_{j}^{-1})^{2}(x)\kappa\bigl((\widetilde{\psi}_{j}^{-1})^{1}(x)\bigr)}
        \\
        &+\int_{\widetilde{\mathcal{U}}_{j}\cap\Omega}\!{}
        \biggl[
        \frac{\partial_{y_{1}}\widetilde{u}_{\tau}\bigl(\widetilde{\psi}_{j}^{-1}(x)\bigr)}
        {1-(\widetilde{\psi}_{j}^{-1})^{2}(x)\kappa\bigl((\widetilde{\psi}_{j}^{-1})^{1}(x)\bigr)}
        +\partial_{y_{2}}\widetilde{u}_{\nu}\bigl(\widetilde{\psi}_{j}^{-1}(x)\bigr)
        \biggr]
        \partial_{y_{2}}\widetilde{v}_{\nu}\bigl(R\widetilde{\psi}_{j}^{-1}(x)\bigr).
    \end{align*}
    Finally, notice that we may rewrite this as
    \begin{align*}
        &\int_{\widetilde{\mathcal{U}}_{j}\cap\Omega}\!{}
        \biggl[
        \frac{\partial_{y_{1}}\widetilde{u}_{\tau}\bigl(\widetilde{\psi}_{j}^{-1}(x)\bigr)}
        {1-(\widetilde{\psi}_{j}^{-1})^{2}(x)\kappa\bigl((\widetilde{\psi}_{j}^{-1})^{1}(x)\bigr)}
        +\partial_{y_{2}}\widetilde{u}_{\nu}\bigl(\widetilde{\psi}_{j}^{-1}(x)\bigr)
        \biggr]
        \frac{\partial_{y_{1}}(\widetilde{v}^{R})_{\tau}\bigl(\widetilde{\psi}_{j}^{-1}(x)\bigr)}
        {1-(\widetilde{\psi}_{j}^{-1})^{2}(x)\kappa\bigl((\widetilde{\psi}_{j}^{-1})^{1}(x)\bigr)}
        \\
        &-\int_{\widetilde{\mathcal{U}}_{j}\cap\Omega}\!{}
        \biggl[
        \frac{\partial_{y_{1}}\widetilde{u}_{\tau}\bigl(\widetilde{\psi}_{j}^{-1}(x)\bigr)}
        {1-(\widetilde{\psi}_{j}^{-1})^{2}(x)\kappa\bigl((\widetilde{\psi}_{j}^{-1})^{1}(x)\bigr)}
        +\partial_{y_{2}}\widetilde{u}_{\nu}\bigl(\widetilde{\psi}_{j}^{-1}(x)\bigr)
        \biggr]
        \partial_{y_{2}}(\widetilde{v}^{R})_{\nu}\bigl(\widetilde{\psi}_{j}^{-1}(x)\bigr)
    \end{align*}
    where, as before, $v^{R}(y)\coloneqq{}v(Ry)$.
    Altogether we find that
    \begin{align*}
        &
        \sum_{i=1}^{2}\kone\int_{\widetilde{\mathcal{U}}_{j}}\!{}
        \bigl<\nabla{}U^{i}(x),\nabla{}v^{i}(x)\bigr>_{j}
        +\ktwo\int_{\widetilde{\mathcal{U}}_{j}}\!{}
        \text{div}_{j}(U)(x)\text{div}_{j}(v)(x)
        \\
        =&\kone\int_{\widetilde{\mathcal{U}}_{j}\cap\Omega}\!{}
        \nabla\widetilde{u}_{\tau}\bigl(\widetilde{\psi}_{j}^{-1}(x)\bigr)^{T}
        \bigl[\mathcal{G}_{j}(x,\widetilde{\psi}_{j}^{-1}(x))\bigr]
        \bigl[\nabla\widetilde{v}_{\tau}\bigl(\widetilde{\psi}_{j}^{-1}(x)\bigr)+
        \nabla\widetilde{v}_{\tau}^{R}\bigl(\widetilde{\psi}_{j}^{-1}(x)\bigr)\bigr]\\
        &+\kone
        \int_{\widetilde{\mathcal{U}}_{j}\cap\Omega}\!{}
        \nabla\widetilde{u}_{\nu}\bigl(\widetilde{\psi}_{j}^{-1}(x)\bigr)^{T}
        \bigl[\mathcal{G}_{j}(x,\widetilde{\psi}_{j}^{-1}(x))\bigr]
        \bigl[\nabla\widetilde{v}_{\nu}\bigl(\widetilde{\psi}_{j}^{-1}(x)\bigr)-
        \nabla\widetilde{v}_{\nu}^{R}\bigl(\widetilde{\psi}_{j}^{-1}(x)\bigr)\bigr]\\
        &+\ktwo\int_{\widetilde{\mathcal{U}}_{j}\cap\Omega}\!{}
        \biggl[
        \frac{\partial_{y_{1}}\widetilde{u}_{\tau}\bigl(\widetilde{\psi}_{j}^{-1}(x)\bigr)}
        {1-(\widetilde{\psi}_{j}^{-1})^{2}(x)\kappa\bigl((\widetilde{\psi}_{j}^{-1})^{1}(x)\bigr)}
        +\partial_{y_{2}}\widetilde{u}_{\nu}\bigl(\widetilde{\psi}_{j}^{-1}(x)\bigr)
        \biggr]
        \frac{\partial_{y_{1}}\widetilde{v}_{\tau}\bigl(\widetilde{\psi}_{j}^{-1}(x)\bigr)+\partial_{y_{1}}(\widetilde{v}^{R})_{\tau}\bigl(\widetilde{\psi}_{j}^{-1}(x)\bigr)}
        {1-(\widetilde{\psi}_{j}^{-1})^{2}(x)\kappa\bigl((\widetilde{\psi}_{j}^{-1})^{1}(x)\bigr)}
        \\
        &+\ktwo\int_{\widetilde{\mathcal{U}}_{j}\cap\Omega}\!{}
        \biggl[
        \frac{\partial_{y_{1}}\widetilde{u}_{\tau}\bigl(\widetilde{\psi}_{j}^{-1}(x)\bigr)}
        {1-(\widetilde{\psi}_{j}^{-1})^{2}(x)\kappa\bigl((\widetilde{\psi}_{j}^{-1})^{1}(x)\bigr)}
        +\partial_{y_{2}}\widetilde{u}_{\nu}\bigl(\widetilde{\psi}_{j}^{-1}(x)\bigr)
        \biggr]
        \bigl[\partial_{y_{2}}\widetilde{v}_{\nu}\bigl(\widetilde{\psi}_{j}^{-1}(x)\bigr)-\partial_{y_{2}}(\widetilde{v}^{R})_{\nu}\bigl(\widetilde{\psi}_{j}^{-1}(x)\bigr)
        \bigr]\\
        &+\int_{\widetilde{\mathcal{U}}_{j}\setminus\Omega}\!{}
        \widetilde{F}_{j}\bigl(x,U(x),\nabla{}U(x)\bigr)\cdot
        v(x),
    \end{align*}
    where $\widetilde{F}_j$ satisfies an estimate similar to \eqref{eq:carathmapApp}.
    We introduce the notation for the even part of a function $v$
    \begin{equation*}
        v_{E}(x)\coloneqq{}\frac{v(x)+\mathfrak{R}_{j}\bigl(x,v(\sigma_{j}(x))\bigr)}{2}.
    \end{equation*}
    Notice that
    \begin{equation*}
        v_{E}(x)=
        \frac{\widetilde{v}_{\tau}\bigl(\widetilde{\psi}_{j}^{-1}(x)\bigr)
        +\widetilde{v}_{\tau}\bigl(R\widetilde{\psi}_{j}^{-1}(x)\bigr)}{2}
        \tau\bigl((\widetilde{\psi}_{j}^{-1})^{1}(x)\bigr)
        +\frac{\widetilde{v}_{\nu}\bigl(\widetilde{\psi}_{j}^{-1}(x)\bigr)
        -\widetilde{v}_{\nu}\bigl(R\widetilde{\psi}_{j}^{-1}(x)\bigr)}{2}
        \nu\bigl((\widetilde{\psi}_{j}^{-1})^{1}(x)\bigr)
    \end{equation*}
    and hence for $x\in\partial\Omega$ we have
    \begin{equation*}
        (v_{E})_{\nu}(x)=
        \frac{\widetilde{v}_{\nu}\bigl(\widetilde{\psi}_{j}^{-1}(x)\bigr)
        -\widetilde{v}_{\nu}\bigl(R\widetilde{\psi}_{j}^{-1}(x)\bigr)}{2}
        =\frac{\widetilde{v}_{\nu}(y_{1},0)-\widetilde{v}_{\nu}(y_{1},0)}{2}=0.
    \end{equation*}
    Because of the above, we can use
    \eqref{eq:TangentialPDEApp}, \eqref{eq:NormalPDEApp}, and that $U=u$ on
    $\Omega$ to obtain
    \begin{equation*}
        \sum_{i=1}^{2}\kone\int_{\widetilde{\mathcal{U}}_{j}}\!{}
        \bigl<\nabla{}U^{i},\nabla{}v^{i}\bigr>_{j}
        +\ktwo\int_{\widetilde{\mathcal{U}}_{j}}\!{}
        \text{div}_{j}(U)\text{div}_{j}(v)
        =2\int_{\widetilde{\mathcal{U}}_{j}\cap\Omega}\!{}
        \frac{U\cdot{}v_{E}}{\varepsilon^{2}}(1-|U|^{2})
        +\int_{\widetilde{\mathcal{U}}_{j}}\!{}
        \widetilde{\mathcal{F}}_{j}\bigl(x,U(x),\nabla{}U(x)\bigr)\cdot
        v(x),
    \end{equation*}
    where $\widetilde{\mathcal{F}}_j(x,U(x),\nabla{}U(x))$ is dependent on $\widetilde{F}_j, F_{j,\nu},F_{j,\tau}$ and satisfies a linear growth condition as in \eqref{eq:carathmapApp}.
    Noting that
    \begin{equation*}
        U(x)\cdot\bigl[\mathfrak{R}_{j}\bigl(x,v(\sigma_{j}(x))\bigr)\bigr]=\mathfrak{R}_{j}(x,U(x))\cdot{}v(\sigma_{j}(x))
        \, ,
    \end{equation*}
    and changing variables gives \eqref{def:ExtendedPDE}.
\qed

\section{Polar example}
\label{app:polar-example}
In this section we illustrate the specific example of the unit disk for which one chart $\mathcal{U}_j$ is sufficient and the curvature is constant.
Here it is easier to see how the distortion factor enters the PDE after reflection at the boundary.
In Remark \ref{app:rem:pullback_metric} we explain that this factor is a natural consequence of the extension and is also present for other choices of metrics such as the standard pullback metric.

We let $\Omega\coloneqq{}B_{1}(0)$ and we describe this set using polar
coordinates.
In order to extend a function defined on $\Omega$ in polar coordinates to
a larger region we introduce the coordinates $y_{1},y_{2}$ defined as
\begin{equation*}
    y_{1}=\theta,\hspace{20pt}
    y_{2}=1-r.
\end{equation*}
Notice that when $r>1$ we have that $y_{2}<0$.
Also, notice that, by Chain Rule, we have
\begin{equation*}
    \partial_{y_{1}}=\partial_{\theta},\hspace{20pt}
    \partial_{y_{2}}=-\partial_{r}.
\end{equation*}
In order to parametrize $\Omega$ we use the map
\begin{equation*}
    \psi(y_{1},y_{2})=(1-y_{2})(\cos(y_{1}),\sin(y_{1}))
\end{equation*}
which has gradient
\begin{equation*}
    \nabla\psi(y_{1},y_{2})=
    \begin{pmatrix}
        -(1-y_{2})\sin(y_{1})& -\cos(y_{1})\\
        (1-y_{2})\cos(y_{1})& -\sin(y_{1})
    \end{pmatrix}
\end{equation*}
and Jacobian
\begin{equation*}
    \det(\nabla\psi(y_{1},y_{2}))=1-y_{2}.
\end{equation*}
Also notice that
\begin{equation*}
    [\nabla\psi(y_{1},y_{2})]^{-1}=
    \begin{pmatrix}
        -\frac{\sin(y_{1})}{1-y_{2}}& \frac{\cos(y_{1})}{1-y_{2}}\\
        -\cos(y_{1})& -\sin(y_{1})
    \end{pmatrix}.
\end{equation*}
On $\Omega$ we use the pullback metric whose associated matrix is given by
\begin{equation*}
    g=\nabla\psi(y_{1},y_{2})^{T}\nabla\psi(y_{1},y_{2})=
    \begin{pmatrix}
        (1-y_{2})^{2}& 0\\
        0& 1
    \end{pmatrix}
\end{equation*}
and whose inverse metric has associated matrix
\begin{equation*}
    g^{-1}=
    \begin{pmatrix}
        \frac{1}{(1-y_{2})^{2}}& 0\\
        0& 1
    \end{pmatrix}.
\end{equation*}
Notice that the Jacobian is given by
\begin{equation*}
    \sqrt{\det(g)}=1-y_{2}.
\end{equation*}
Relative to the metric $g$ we observe that the normalized differential operators are
\begin{equation*}
    E_{1}\coloneqq\frac{1}{1-y_{2}}\partial_{y_{1}}\hspace{20pt}
    E_{2}\coloneqq\partial_{y_{2}}.
\end{equation*}
If a vector field $F$ is described as
\begin{equation*}
    F(x)=F_{\tau}(x)(-\sin((\psi^{-1})^{1}(x))),\cos((\psi^{-1})^{1}(x)))^{T}
    +F_{\nu}(x)(-\cos((\psi^{-1})^{1}(x))),-\sin((\psi^{-1})^{1}(x))))^{T}
\end{equation*}
then, since $\mathrm{d}\psi_{\#}(\mathbf{e}_{1})=(1-y_{2})(-\sin(y_{1}),\cos(y_{1}))$ and
$\mathrm{d}\psi_{\#}(\mathbf{e}_{2})=(-\cos(\theta),-\sin(\theta))$, we see that this pulls back under $\psi$ as
\begin{equation*}
    [\psi^{*}F](y)=\frac{\widetilde{F}_{\tau}(y)}{1-y_{2}}\mathbf{e}_{1}
    +\widetilde{F}_{\nu}(y)\mathbf{e}_{2}
\end{equation*}
where $\widetilde{F}_{\tau}$ and $\widetilde{F}_{\nu}$ denote the coordinate descriptions of
$F_{\tau}$ and $F_{\nu}$ respectively.
The divergence, relative to $g$, of this pullback vector field is given by
\begin{equation*}
    \text{div}_{g}\bigl([\psi^{*}F]\bigr)
    =\frac{\partial_{y_{1}}\widetilde{F}_{\tau}(y)}{1-y_{2}}
    +\frac{\partial_{y_{2}}\Bigl[(1-y_{2})\widetilde{F}_{\nu}(y)\Bigr]}{1-y_{2}}.
\end{equation*}
Finally, for a scalar function, $f$, we set
\begin{equation*}
    \nabla_{g}f\coloneqq{}g^{-1}\mathrm{d}f.
\end{equation*}
On the exterior of $\Omega$ we continue to use $\psi$ as a means of defining coordinates.
This means that we parametrize $B_{2}(0)\setminus\Omega$ using
$y_{1}\in(0,2\pi)$ and $y_{2}<0$.
We also intend to use the metric $\overline{g}$ whose associated matrix is given by
\begin{equation*}
    \overline{g}=\nabla\psi(y_{1},-y_{2})^{T}\nabla\psi(y_{1},-y_{2})
    =
    \begin{pmatrix}
        (1+y_{2})^{2}& 0\\
        0& 1
    \end{pmatrix}.
\end{equation*}
Notice that this is not the pullback metric according to the coordinates $\psi$.
In addition, notice that this is only non-singular on $B_{2}(0)$.
The associated inverse matrix is
\begin{equation*}
    \overline{g}^{-1}=\nabla\psi(y_{1},-y_{2})^{T}\nabla\psi(y_{1},-y_{2})
    =
    \begin{pmatrix}
        \frac{1}{(1+y_{2})^{2}}& 0\\
        0& 1
    \end{pmatrix}.
\end{equation*}
We notice that
\begin{equation*}
    \sqrt{\det(\overline{g})}=1+y_{2}
\end{equation*}
and the normalized differential operators realtive to $\overline{g}$ are
\begin{equation*}
    \overline{E}_{1}\coloneqq\frac{1}{1+y_{2}}\partial_{y_{1}},\hspace{20pt}
    \overline{E}_{2}\coloneqq\partial_{y_{2}}.
\end{equation*}
If a vector field $F$ is described as
\begin{align*}
    F(x)=&\widetilde{F}_{\tau}((\psi^{-1})^{1}(x),-(\psi^{-1})^{2}(x))(-\sin((\psi^{-1})^{1}(x))),\cos((\psi^{-1})^{1}(x)))^{T}\\
    &-\widetilde{F}_{\nu}((\psi^{-1})^{1}(x),-(\psi^{-1})^{2}(x))(-\cos((\psi^{-1})^{1}(x))),-\sin((\psi^{-1})^{1}(x))))^{T}
\end{align*}
then we see that this pulls back under $\psi$ as
\begin{equation*}
    [\psi^{*}F](y)=\frac{\widetilde{F}_{\tau}(y_{1},-y_{2})}{1-y_{2}}\mathbf{e}_{1}
    -\widetilde{F}_{\nu}(y_{1},-y_{2})\mathbf{e}_{2}.
\end{equation*}
The divergence, relative to $\overline{g}$, of this pullback vector field is given by
\begin{equation*}
    \text{div}_{\overline{g}}\bigl([\psi^{*}F]\bigr)
    =\frac{1}{1+y_{2}}
    \partial_{y_{1}}\biggl(\frac{1+y_{2}}{1-y_{2}}\widetilde{F}_{\tau}(y_{1},y_{2})\biggr)
    -\frac{\partial_{y_{2}}\Bigl[(1+y_{2})\widetilde{F}_{\nu}(y_{1},-y_{2})\Bigr]}{1+y_{2}}.
\end{equation*}
Notice that when considering only the highest order terms of the divergence we obtain
\begin{equation*}
    \frac{\partial_{y_{1}}\widetilde{F}_{\tau}(y_{1},-y_{2})}{1-y_{2}}
    +\partial_{y_{2}}\widetilde{F}_{\nu}(y_{1},-y_{2})
\end{equation*}
where $y_{2}<0$.
Since the argument of the coordinate functions is $y_{1},-y_{2}$ we will need to make a
change of coordinates.
Notice that for this to end up with a factor of $\frac{1}{1-y_{2}}$ after this change of
variables we need that the first term be multiplied by
\begin{equation*}
    \frac{1-y_{2}}{1+y_{2}}.
\end{equation*}
We require this to obtain a PDE of similar form as inside the interior.
As a result, we will be interested in
\begin{equation*}
    \frac{1-y_{2}}{(1+y_{2})^{2}}
    \partial_{y_{1}}\biggl(\frac{1+y_{2}}{1-y_{2}}\widetilde{F}_{\tau}(y_{1},y_{2})\biggr)
    -\frac{\partial_{y_{2}}\Bigl[(1+y_{2})\widetilde{F}_{\nu}(y_{1},-y_{2})\Bigr]}{1+y_{2}}
\end{equation*}
which does not agree with the divergence with respect to the metric $\overline{g}$.
From this we see that a distortion factor is necessary to obtain the desired PDE after
reflecting.
\begin{remark}\label{app:rem:pullback_metric}
    Notice that if we had used the standard pullback metric
    \begin{equation*}
        \overline{g}=\nabla\psi(y)^{T}\nabla\psi(y)
        =
        \begin{pmatrix}
            (1-y_{2})^{2}& 0\\
            0& 1
        \end{pmatrix}
    \end{equation*}
    there would still be distortion factors.
    To see this, observe that in this case $\sqrt{\det(\overline{g})}=1-y_{2}$ and that
    our pullback vector field will be
    \begin{equation*}
    [\psi^{*}F](y)=\frac{\widetilde{F}_{\tau}(y_{1},-y_{2})}{1-y_{2}}\mathbf{e}_{1}
    -\widetilde{F}_{\nu}(y_{1},-y_{2})\mathbf{e}_{2}.
    \end{equation*}
    Computing the divergence with this metric gives
    \begin{equation*}
        \text{div}_{\overline{g}}([\psi^{*}F])=
        \frac{1}{1-y_{2}}\partial_{y_{1}}\widetilde{F}_{\tau}(y_{1},-y_{2})
        +\frac{1}{1-y_{2}}\partial_{y_{2}}\Bigl((1-y_{2})\widetilde{F}_{\nu}(y_{1},-y_{2})\Bigr)
    \end{equation*}
    which still requires a corrective factor for the partial derivative in $y_{1}$.
    The benefit of using the pullback metric is that there is a corresponding correct factor
    needed for the gradient.
    To see this notice that to highest order
    \begin{align*}
        \bigl<\nabla_{g}u(y_{1},-y_{2}),\nabla_{g}v(y)\bigr>
        &=g^{-1}(\mathrm{d}u(y_{1},-y_{2}),\mathrm{d}v(y))\\
        &\approx\frac{\partial_{y_{1}}\widetilde{u}_{\tau}
        \partial_{y_{1}}\widetilde{v}_{\tau}}{(1-y_{2})^{2}}
        +\partial_{y_{2}}\widetilde{u}_{\tau}\partial_{y_{2}}\widetilde{v}_{\tau}
        +\frac{\partial_{y_{1}}\widetilde{u}_{\nu}
        \partial_{y_{1}}\widetilde{v}_{\nu}}{(1-y_{2})^{2}}
        +\partial_{y_{2}}\widetilde{u}_{\nu}\partial_{y_{2}}\widetilde{v}_{\nu}
    \end{align*}
    and so a correction factor is needed in order to have a denominator of
    $\frac{1}{(1+y_{2})^{2}}$ before reflecting.
\end{remark}

\paragraph{Acknowledgements.}
LB acknowledges support from an NSERC (Canada) Discovery Grant.

\bibliographystyle{acm}
\bibliography{Bibliography}

\end{document}